\documentclass[a4paper,12pt]{article}

\usepackage[french]{babel}
\usepackage[utf8]{inputenc}
\usepackage[T1]{fontenc}
\usepackage{lmodern}
\usepackage{amsmath,amssymb,amsthm}
\usepackage{enumerate}
\usepackage{tikz-cd}
\usepackage[]{hyperref}

\DeclareFontFamily{OMX}{lmex}{}
\DeclareFontShape{OMX}{lmex}{m}{n}{<->lmex10}{}

\tikzcdset{arrow style=Latin Modern}

\theoremstyle{plain}
\newtheorem{theo}{Théorème}[section]
\newtheorem{prop}[theo]{Proposition}

\newtheorem{coro}[theo]{Corollaire}
\newtheorem{lemm}[theo]{Lemme}
\newtheorem*{slem}{Sous-lemme}
\theoremstyle{definition}
\newtheorem{defi}[theo]{Définition}
\theoremstyle{remark}
\newtheorem{rema}[theo]{Remarque}
\newtheorem{exem}[theo]{Exemple}

\DeclareMathOperator{\im}{im}
\DeclareMathOperator{\card}{card}
\DeclareMathOperator{\diag}{diag}
\DeclareMathOperator{\Hom}{Hom}
\DeclareMathOperator{\Ext}{Ext}
\DeclareMathOperator{\Lie}{Lie}
\DeclareMathOperator{\detfr}{d\acute{e}t}
\newcommand{\llbrack}{[\![}
\newcommand{\rrbrack}{]\!]}
\newcommand{\dfn}{\overset{\text{déf}}{=}}
\newcommand{\iso}{\overset{\sim}{\longrightarrow}}
\newcommand{\un}{\underline{1}}
\newcommand{\unc}{\widehat{\un}}
\newcommand{\der}{\mathrm{der}}
\newcommand{\Nbb}{\mathbb{N}}
\newcommand{\Z}{\mathbb{Z}}
\newcommand{\Q}{\mathbb{Q}}
\newcommand{\Fp}{\mathbb{F}_p}
\newcommand{\Zp}{\Z_p}
\newcommand{\Qp}{\Q_p}
\newcommand{\Oe}{\mathcal{O}_{\!E}}
\newcommand{\pe}{\varpi_{\!E}}
\newcommand{\ke}{k_E}
\newcommand{\A}[1]{\Oe/\pe^{#1}\Oe}
\newcommand{\val}{\operatorname{val}_p}
\newcommand{\Nrm}{\operatorname{N}_{F/\Qp}}

\newcommand{\GL}{\mathrm{GL}}

\newcommand{\SL}{\mathrm{SL}}
\newcommand{\PGL}{\mathrm{PGL}}
\newcommand{\St}{\mathrm{St}}
\newcommand{\Sp}{\mathrm{Sp}}
\newcommand{\Stc}{\widehat{\mathrm{St}}}
\newcommand{\Spc}{\widehat{\mathrm{Sp}}}
\newcommand{\Spcz}{\Spc{}^0}
\newcommand{\epsa}{\varepsilon^{-1} \circ \alpha}

\newcommand{\oma}{\omega^{-1} \circ \alpha}
\newcommand{\omb}{\omega^{-1} \circ \beta}

\newcommand{\Ga}{G_\alpha}
\newcommand{\Ba}{B_\alpha}
\newcommand{\Pa}{P_\alpha}
\newcommand{\Za}{Z_\alpha}
\newcommand{\NPz}{N_{P,0}}
\newcommand{\Na}{N_\alpha}
\newcommand{\Naz}{N_{\alpha,0}}
\newcommand{\E}{\mathcal{E}}
\newcommand{\Ea}{\E_\alpha}
\newcommand{\Hc}[1][\bullet]{\mathrm{H}^{#1}}
\newcommand{\Rc}[1][\bullet]{\mathrm{R}^{#1}}
\newcommand{\fin}[1]{#1\!\operatorname{-fin}}
\newcommand{\Ord}[1][B]{\operatorname{Ord}_{#1(F)}}
\newcommand{\HOrd}[1][\bullet]{\Hc[#1]\!\Ord}
\newcommand{\HOrdP}[1][\bullet]{\Hc[#1]\!\Ord[P]}
\newcommand{\OrdQp}[1][B]{\operatorname{Ord}_{#1(\Qp)}}
\newcommand{\HOrdQp}[1][B]{\operatorname{H^1 Ord}_{#1(\Qp)}}
\newcommand{\ROrdQp}[1][\bullet]{\Rc[#1]\!\OrdQp}

\newcommand{\Ind}[1][B^-]{\operatorname{Ind}^{G(F)}_{#1(F)}}
\newcommand{\IndL}[1][B_L^-]{\operatorname{Ind}^{L(F)}_{#1(F)}}
\newcommand{\IndQp}[1][B^-]{\operatorname{Ind}^{G(\Qp)}_{#1(\Qp)}}
\newcommand{\IndQpL}[1][B_L^-]{\operatorname{Ind}^{L(\Qp)}_{#1(\Qp)}}
\newcommand{\Inda}[1][\Ba^-]{\operatorname{Ind}^{\Ga(\Qp)}_{#1(\Qp)}}
\newcommand{\Indaprime}[1][(\Ba^- \cap \Ga')]{\operatorname{Ind}^{\Ga'(\Qp)}_{#1(\Qp)}}
\newcommand{\IndQpH}{\operatorname{Ind}^{\Ga'(\Qp)}_{(\Ba^- \cap \Ga')(\Qp)}}
\newcommand{\IndQpGLd}{\operatorname{Ind}^{\GL_2(\Qp)}_{B_2^-(\Qp)}}
\newcommand{\IndQpSLd}{\operatorname{Ind}^{\SL_2(\Qp)}_{(B_2^- \cap \SL_2)(\Qp)}}
\newcommand{\IndQpPGLd}{\operatorname{Ind}^{\PGL_2(\Qp)}_{(B_2^-/Z_2)(\Qp)}}
\newcommand{\Clis}{\mathcal{C}^\infty}
\newcommand{\w}{\widetilde{w}}
\newcommand{\W}{\widetilde{W}}

\hypersetup{pdfinfo={
	Title={Compléments sur les extensions entre séries principales p-adiques et modulo p de G(F)},
	Author={Julien Hauseux},
	Subject={22E50},
	Keywords={extensions, séries principales, parties ordinaires, filtration de Bruhat}
}}

\title{Compléments sur les extensions entre séries principales $p$-adiques et modulo $p$ de $G(F)$}
\author{Julien Hauseux}
\date{}

\begin{document}

\maketitle

\begin{abstract}
Nous complétons les résultats de \cite{JH}.
Soit $G$ un groupe réductif connexe déployé sur une extension finie $F$ de $\mathbb{Q}_p$. Lorsque $F=\mathbb{Q}_p$, nous déterminons les extensions entre séries principales $p$-adiques et modulo $p$ de $G(\mathbb{Q}_p)$ sans supposer le centre de $G$ connexe ou le groupe dérivé de $G$ simplement connexe. Cela fait apparaître un phénomène nouveau : il peut exister plusieurs extensions non scindées non isomorphes entre deux séries principales distinctes.
Nous complétons aussi les calculs d'auto-extensions d'une série principale dans les cas non génériques lorsque le centre de $G$ est connexe.
Nous déterminons enfin les extensions d'une série principale de $G(F)$ par une représentation \og ordinaire \fg{} de $G(F)$ (c'est-à-dire obtenue par induction parabolique à partir d'une représentation spéciale tordue par un caractère). Pour cela, nous calculons le $\delta$-foncteur $\mathrm{H^\bullet Ord}_{B(F)}$ des parties ordinaires dérivées d'Emerton relatif à un sous-groupe de Borel sur une représentation ordinaire de $G(F)$.
\end{abstract}

\tableofcontents

\section{Introduction}

\subsection*{Contexte}

Nous rappelons les résultats de \cite{JH}. Soient $F$ une extension finie de $\Qp$ et $G$ un groupe réductif connexe déployé sur $F$. On fixe une extension finie $E$ de $\Qp$.

Lorsque $F=\Qp$, nous avons calculé les extensions entre séries principales continues unitaires de $G(\Qp)$ sur $E$ en faisant les hypothèses suivantes sur $G$ : son centre est connexe et son groupe dérivé est simplement connexe. Dans cet article, nous traitons le cas général et nous complétons les calculs d'auto-extensions d'une série principale dans les cas non génériques lorsque le centre de $G$ est connexe.

Nous avons également calculé les extensions d'une série principale continue unitaire de $G(F)$ sur $E$ par l'induite parabolique d'un caractères continu unitaire. Dans cet article, nous généralisons ces calculs pour les représentations continues unitaires \og ordinaires \fg{} de $G(F)$ sur $E$ (c'est-à-dire obtenues par induction parabolique à partir d'une représentation continue unitaire spéciale tordue par un caractère continu unitaire).

\subsection*{Principaux résultats}

Soient $B \subset G$ un sous-groupe de Borel et $T \subset B$ un tore maximal déployé. On note $B^- \subset G$ le sous-groupe de Borel opposé à $B$ par rapport à $T$, $\Delta$ les racines simples de $(G,B,T)$ et pour tout $\alpha \in \Delta$, on note $s_\alpha$ la réflexion simple correspondante. On note $\varepsilon : F^\times \to \Zp^\times$ le caractère cyclotomique $p$-adique et $\Oe$ l'anneau des entiers de $E$.
On calcule les $\Ext^1$ dans les catégories abéliennes de représentations continues unitaires admissibles sur $E$ en utilisant les extensions de Yoneda.

\medskip

Nous calculons tout d'abord les extensions entre séries principales continues unitaires de $G(F)$ sur $E$.
Lorsque $F \neq \Qp$, ces extensions proviennent toujours d'une extension entre caractères de $T(F)$ (voir \cite[Théorème 1.2]{JH}). On suppose donc $F=\Qp$ et on généralise \cite[Théorème 1.1]{JH}.

En comparaison si le groupe dérivé de $G$ n'est pas simplement connexe, alors $G$ n'admet pas nécessairement un \og twisting element \fg{} $\theta$ (par exemple $G=\PGL_2$) et si le centre de $G$ n'est pas connexe, alors on peut avoir
\begin{equation*}
\card \left\{ \alpha \in \Delta \mid \chi' = s_\alpha(\chi) \cdot (\epsa) \right\}>1
\end{equation*}
avec $\chi,\chi' : T(\Qp) \to \Oe^\times \subset E^\times$ des caractères continus unitaires distincts (voir l'exemple \ref{exem:card}).

Nous démontrons le résultat suivant (Théorème \ref{theo:ext}), ainsi que son analogue modulo $p$ (c'est-à-dire dans les catégories de représentations lisses admissibles sur le corps résiduel $\ke$ de $E$).

\begin{theo} \label{theo:1}
Soit $\chi : T(\Qp) \to \Oe^\times \subset E^\times$ un caractère continu unitaire.
\begin{enumerate}[(i)]
\item Si $\chi' : T(\Qp) \to \Oe^\times \subset E^\times$ est un caractère continu unitaire distinct de $\chi$, alors
\begin{multline*}
\dim_E \Ext^1_{G(\Qp)} \left( \IndQp \chi',\IndQp \chi \right) \\
= \card \left\{ \alpha \in \Delta \mid \chi' = s_\alpha(\chi) \cdot (\epsa) \right\}.
\end{multline*}
\item Si $s_\alpha(\chi) \cdot (\epsa) \neq \chi$ pour tout $\alpha \in \Delta$, alors le foncteur $\IndQp$ induit un isomorphisme $E$-linéaire
\begin{equation*}
\Ext_{T(\Qp)}^1 \left( \chi, \chi \right) \iso \Ext_{G(\Qp)}^1 \left( \IndQp \chi,\IndQp \chi \right).
\end{equation*}
\end{enumerate}
\end{theo}

Sans l'hypothèse de généricité du point (ii), le foncteur $\IndQp$ induit une injection $E$-linéaire et on donne un minorant et un majorant de la dimension de son conoyau (voir le point (ii) de la remarque \ref{rema:ext}).

Lorsque le centre de $G$ est connexe, nous complétons les calculs d'auto-extensions d'une série principale modulo $p$ dans les cas non génériques (Théorème \ref{theo:autoextmodp}) et nous en déduisons le résultat analogue $p$-adique lorsque $p \neq 2$ (Corollaire \ref{coro:autoext}).

\begin{theo} \label{theo:2}
On suppose le centre de $G$ connexe. Soit $\chi : T(\Qp) \to \ke^\times$ un caractère lisse.
\begin{enumerate}[(i)]
\item Si $p \neq 2$, alors le foncteur $\IndQp$ induit un isomorphisme $\ke$-linéaire
\begin{equation*}
\Ext_{T(\Qp)}^1 \left( \chi,\chi \right) \iso \Ext_{G(\Qp)}^1 \left( \IndQp \chi,\IndQp \chi \right).
\end{equation*}
\item Si $p=2$, alors le foncteur $\IndQp$ induit une injection $\ke$-linéaire
\begin{equation*}
\Ext_{T(\Qp)}^1 \left( \chi,\chi \right) \hookrightarrow \Ext_{G(\Qp)}^1 \left( \IndQp \chi,\IndQp \chi \right)
\end{equation*}
dont le conoyau est de dimension $\card \{\alpha \in \Delta \mid s_\alpha(\chi)=\chi\}$.
\end{enumerate}
\end{theo}

\begin{coro}
On suppose le centre de $G$ connexe et $p \neq 2$.
Pour tout caractère continu unitaire $\chi : T(\Qp) \to \Oe^\times \subset E^\times$, le foncteur $\IndQp$ induit un isomorphisme $E$-linéaire
\begin{equation*}
\Ext_{T(\Qp)}^1 \left( \chi,\chi \right) \iso \Ext_{G(\Qp)}^1 \left( \IndQp \chi,\IndQp \chi \right).
\end{equation*}
\end{coro}

Nous étendons enfin nos calculs aux extensions d'une série principale de $G(F)$ par une représentation ordinaire de $G(F)$ (Propositions \ref{prop:extord} et \ref{prop:extordF}).

\subsection*{Méthodes utilisées}

Expliquons la preuve du théorème \ref{theo:1}. On suppose $F=\Qp$. Soient $\chi,\chi' : T(\Qp) \to \Oe^\times \subset E^\times$ des caractères continus unitaires distincts et $\Delta' \subset \Delta$ le sous-ensemble des racines simples $\alpha \in \Delta$ telles que $\chi' = s_\alpha(\chi) \cdot (\epsa)$.
Pour tout $\alpha \in \Delta$, on note $\Ga \subset G$ le centralisateur de $(\ker \alpha)^\circ \subset T$. Dans la sous-section \ref{ssec:ga}, on montre que pour tout $\alpha \in \Delta'$ il existe une extension non scindée
\begin{equation*}
0 \to \Inda[(B^- \cap \Ga)] \chi \to \Ea \to \Inda[(B^- \cap \Ga)] \chi' \to 0.
\end{equation*}
Le résultat principal de la sous-section \ref{ssec:compat} permet alors de montrer que les classes des extensions
\begin{equation*}
\left( \IndQp[(B^-\Ga)] \Ea \right)_{\alpha \in \Delta'}
\end{equation*}
sont linéairement indépendantes dans le $E$-espace vectoriel du point (i), d'où une première inégalité. La seconde inégalité et le point (ii) se démontrent par réduction modulo $p^k$ et dévissage en utilisant les calculs de parties ordinaires dérivées de \cite[§ 4.2]{JH} et la suite exacte dérivée de la relation d'adjonction entre les foncteurs $\OrdQp$ et $\IndQp$ (voir la sous-section \ref{ssec:preli}).

De même, le théorème \ref{theo:2} se démontre en se ramenant au cas de $\GL_2(\Qp)$, mais en utilisant la suite exacte complète de cinq termes.

\medskip

Expliquons les preuves des proposition \ref{prop:extord} et \ref{prop:extordF}. On procède par réduction modulo $p^k$ et dévissage comme dans \cite{JH}. En caractéristique positive, on utilise le $\delta$-foncteur $\HOrd$ des parties ordinaires dérivées d'Emerton. On fixe une uniformisante $\pe$ de $\Oe$ et un entier $k \geq 1$.
On définit tout d'abord la filtration de Bruhat d'une représentation ordinaire de $G(F)$ sur $\A{k}$, puis on calcule son gradué qui est un facteur direct du gradué de la filtration de Bruhat d'une série principale de $G(F)$ sur $\A{k}$. On utilise ensuite les calculs de parties ordinaires dérivées de \cite[§ 4.3]{JH} pour en déduire l'expression du $\delta$-foncteur $\HOrd$ sur une telle représentation (Théorème \ref{theo:hord}). On explicite enfin ce calcul en degrés $0$ et $1$ (Corollaires \ref{coro:ord} et \ref{coro:h1ord}) et on conclut en utilisant la suite exacte dérivée de la relation d'adjonction entre les foncteurs $\Ord$ et $\Ind$ (voir la sous-section \ref{ssec:preli}).

\subsection*{Notations et conventions}

Soit $F$ une extension finie de $\Qp$. On note $\varepsilon : F^\times \to \Zp^\times$ le caractère cyclotomique $p$-adique (défini par $\varepsilon(x) = \Nrm(x) \lvert \Nrm(x) \rvert_p$ pour tout $x \in F^\times$) et $\omega : F^\times \to \Fp^\times$ sa réduction modulo $p$.

Soient $G$ un groupe réductif connexe déployé sur $F$, $B \subset G$ un sous-groupe de Borel et $T \subset B$ un tore maximal déployé. On note $B^- \subset G$ le sous-groupe de Borel opposé à $B$ par rapport à $T$ et $N$ le radical unipotent de $B$.
On note $W$ le groupe de Weyl de $(G,T)$, $\ell : W \to \Nbb$ la longueur relative à $B$ et $\Delta$ les racines simples de $(G,B,T)$. Pour tout $\alpha \in \Delta$, on note $s_\alpha \in W$ la réflexion simple correspondante et pour tout $w \in W$, on note $\dot{w} \in G(F)$ un représentant de $w$ dans le normalisateur de $T(F)$.

Si $P \subset G$ est un sous-groupe parabolique standard (c'est-à-dire contenant $B$) et $L \subset P$ est le sous-groupe de Levi standard (c'est-à-dire contenant $T$), on note $P^- \subset G$ le sous-groupe parabolique opposé à $P$ par rapport à $L$, $B_L \subset L$ (resp. $B_L^- \subset L$) le sous-groupe de Borel $B \cap L$ (resp. $B^- \cap L$) et $\Delta_L \subset \Delta$ les racines simples de $(L,B_L,T)$.

Soit $E$ une extension finie de $\Qp$. On note $\Oe$ l'anneau des entiers de $E$ et $\ke$ le corps résiduel de $\Oe$. On fixe une uniformisante $\pe$ de $\Oe$. On désigne par $A$ une $\Oe$-algèbre locale artinienne de corps résiduel $\ke$ et on note encore $\omega : F^\times \to A^\times$ l'image de $\varepsilon$ dans $A^\times$.

Pour les représentations d'un groupe de Lie $p$-adique, on utilise la terminologie de \cite[§ 2]{Em1} pour les représentations lisses à coefficients dans $A$ et on renvoie à \cite[§ 3.1]{BH} pour les représentations continues unitaires admissibles sur des $E$-espaces de Banach.

\subsection*{Remerciements}

Je remercie chaleureusement mon directeur de thèse Christophe Breuil d'avoir suivi ce travail avec attention, ainsi que pour ses remarques et ses conseils.
Je remercie également Florian Herzig et le rapporteur anonyme pour de nombreux commentaires qui ont permis d'améliorer cet article.

\numberwithin{theo}{subsection}

\section{Parties ordinaires dérivées}

Nous commençons par faire quelques rappels sur le $\delta$-foncteur $\HOrdP$. Puis, nous démontrons une compatibilité naturelle lorsque $F=\Qp$ entre les calculs de parties ordinaires dérivées de \cite{JH} en degré $1$ et l'induction à partir d'un sous-groupe parabolique correspondant à une racine simple. Enfin, nous généralisons les résultats de \cite[§ 2.3]{JH} et \cite[§ 4.3]{JH} aux représentations ordinaires.

\subsection{Préliminaires} \label{ssec:preli}

Soient $P \subset G$ un sous-groupe parabolique standard et $L \subset P$ le sous-groupe de Levi standard. On note $N_P$ le radical unipotent de $P$ et $Z_L$ le centre de $L$. On fixe un sous-groupe ouvert compact $\NPz$ de $N_P(F)$ et on définit des sous-monoïdes de $L(F)$ et $Z_L(F)$ en posant
\begin{gather*}
L^+ \dfn \left\{ l \in L(F) \mid l \NPz l^{-1} \subset \NPz \right\}, \\
Z_L^+ \dfn Z_L(F) \cap L^+.
\end{gather*}

On rappelle tout d'abord la construction du $\delta$-foncteur $\HOrdP$.
Soit $V$ une représentation lisse de $P(F)$ sur $A$. D'après \cite[§ 3]{Em2}, les $A$-modules $\Hc(\NPz,V)$ sont naturellement munis d'une \emph{action de Hecke} de $L^+$ et les $A$-modules
\begin{equation*}
\HOrdP V \dfn \Hom_{A[Z_L^+]} \left( A[Z_L(F)],\Hc(\NPz,V) \right)_{\fin{Z_L(F)}}
\end{equation*}
sont naturellement des représentations lisses de $L(F)$ sur $A$, appelées les \emph{parties ordinaires dérivées} de $V$. De plus si $V$ est une représentation lisse localement admissible de $G(F)$ sur $A$, alors $\HOrdP V$ est le \emph{localisé en $Z_L^+$ de $\Hc(\NPz,V)$} : on a un isomorphisme naturel $L(F)$-équivariant
\begin{equation*}
\HOrdP V \cong A[Z_L(F)] \otimes_{A[Z_L^+]} \Hc(\NPz,V)
\end{equation*}
(voir \cite[Lemme 3.2.1]{Em2} et \cite[Théorème 3.4.7]{Em2}).

On rappelle maintenant la construction de l'isomorphisme naturel qui fait du foncteur des parties ordinaires $\Ord[P] = \HOrdP[0]$ un quasi-inverse à gauche du foncteur d'induction parabolique $\Ind[P^-]$.

\begin{lemm} \label{lemm:inj}
Soit $U$ une représentation lisse localement admissible de $L(F)$ sur $A$. On a une injection naturelle $L^+$-équivariante
\begin{equation*}
U \hookrightarrow \left( \Ind[P^-] U \right)^{\NPz}
\end{equation*}
qui induit après localisation en $Z_L^+$ un isomorphisme $L(F)$-équivariant
\begin{equation*}
U \iso \Ord[P] \left( \Ind[P^-] U \right).
\end{equation*}
\end{lemm}

\begin{proof}
La projection $G(F) \twoheadrightarrow P^-(F) \backslash G(F)$ induit une immersion ouverte $N_P(F) \hookrightarrow P^-(F) \backslash G(F)$, d'où une injection naturelle $P(F)$-équivariante (voir \cite[Lemme 4.1.9]{Em1})
\begin{equation*}
\Clis_c(N_P(F),U) \hookrightarrow \Ind[P^-] U
\end{equation*}
qui induit un isomorphisme $L(F)$-équivariant (voir \cite[Lemme 4.3.1]{Em1})
\begin{equation*}
\Ord[P] \left( \Clis_c(N_P(F),U) \right) \iso \Ord[P] \left( \Ind[P^-] U \right).
\end{equation*}
On déduit de \cite[Proposition 4.2.7]{Em1} que l'injection naturelle $A$-linéaire
\begin{equation*}
U \hookrightarrow \Clis_c(N_P(F),U)^{\NPz}
\end{equation*}
définie par $u \mapsto 1_{\NPz} u$ avec $1_{\NPz}$ la fonction caractéristique de $\NPz$ sur $N_P(F)$ est $L^+$-équivariante et induit après localisation en $Z_L^+$ un isomorphisme $L(F)$-équivariant
\begin{equation*}
U \iso \Ord[P] \left( \Clis_c(N_P(F),U) \right).
\end{equation*}
En composant ces deux injections, on obtient une injection naturelle qui vérifie bien la propriété de l'énoncé.
\end{proof}

On rappelle enfin que la relation d'adjonction entre les foncteurs $\Ind[P^-]$ et $\Ord[P]$ induit une suite exacte de $A$-modules
\begin{multline} \label{SEextP}
0 \to \Ext^1_{L(F)} \left( U,\Ord[P] V \right) \to \Ext^1_{G(F)} \left( \Ind[P^-] U,V \right) \\
\to \Hom_{L(F)} \left( U,\HOrdP[1] V \right)
\end{multline}
pour toutes représentations lisses localement admissibles $U$ et $V$ de $L(F)$ et $G(F)$ respectivement sur $A$ (voir \cite[§ 3.7]{Em2}). Le premier morphisme non trivial est induit par le foncteur exact $\Ind[P^-]$ et le morphisme naturel $\Ind[P^-](\Ord[P] V) \to V$ et le second associe à la classe d'une extension
\begin{equation*}
0 \to V \to \E \to \Ind[P^-] U \to 0
\end{equation*}
le morphisme $\delta$ de la suite exacte longue
\begin{equation*}
0 \to \Ord[P] V \to \Ord[P] \E \to \Ord[P] \left( \Ind[P^-] U \right) \overset{\delta}{\longrightarrow} \HOrdP[1] V
\end{equation*}
composé avec l'inverse de l'isomorphisme du lemme \ref{lemm:inj}.

\subsection{\texorpdfstring{Compatibilité avec l'induction en degré $1$}{Compatibilité avec l'induction en degré 1}} \label{ssec:compat}

On suppose $F =\Qp$ et on fixe $\alpha \in \Delta$.
On note $\Ga \subset G$ le sous-groupe fermé engendré par $T$ et les sous-groupes radiciels correspondant aux racines $\pm \alpha$.
On note $\Pa \subset G$ (resp. $\Pa^- \subset G$) le sous-groupe parabolique $B\Ga$ (resp. $B^-\Ga$) et $\Na \subset \Pa$ son radical unipotent.
On note $\Ba \subset \Ga$ (resp. $\Ba^- \subset \Ga$) le sous-groupe de Borel $B \cap \Ga$ (resp. $B^- \cap \Ga$) et $\Na'' \subset \Ba$ son radical unipotent.
On note enfin $\Za \subset T$ le centre de $\Ga$.

\medskip

Soient $U$ et $V$ des représentations lisses localement admissibles de $T(\Qp)$ sur $A$.
D'après le lemme \ref{lemm:inj} pour les triplets $(G,B,T)$ et $(\Ga,\Ba,T)$, on a des isomorphismes naturels $T(\Qp)$-équivariants
\begin{gather}
\OrdQp \left( \IndQp U \right) \cong U, \label{isoOrd} \\
\OrdQp[\Ba] \left( \Inda U \right) \cong U \label{isoOrda}.
\end{gather}
Pour tout $\beta \in \Delta$, on note $V^\beta$ la représentation lisse de $T(\Qp)$ sur $A$ dont le $A$-module sous-jacent est $V$ et sur lequel $t \in T(\Qp)$ agit à travers $s_\beta(t)$.
D'après \cite[Corollaire 4.2.4 (i)]{JH} pour les triplets $(G,B,T)$ et $(\Ga,\Ba,T)$, on a des isomorphismes naturels $T(\Qp)$-équivariants
\begin{gather}
\HOrdQp \left( \IndQp V \right) \cong \bigoplus_{\beta \in \Delta} V^\beta \otimes (\omb), \label{isoHOrd} \\
\HOrdQp[\Ba] \left( \Inda V \right) \cong V^\alpha \otimes (\oma). \label{isoHOrda}
\end{gather}
Le but de cette sous-section est de démontrer le résultat suivant.

\begin{prop} \label{prop:compat}
On a un diagramme commutatif de $A$-modules
\begin{equation*} \begin{tikzcd}[column sep=tiny,row sep=small]
\Ext^1_{G(\Qp)} \left( \IndQp U,\IndQp V \right) \rar & \displaystyle \bigoplus_{\beta \in \Delta} \Hom_{T(\Qp)} \left( U,V^\beta \otimes (\omb) \right) \\
\Ext^1_{\Ga(\Qp)} \left( \Inda U,\Inda V \right) \uar[hook] \rar & \Hom_{T(\Qp)} \left( U,V^\alpha \otimes (\oma) \right) \uar[hook]
\end{tikzcd} \end{equation*}
où les morphismes horizontaux sont donnés par le second morphisme non trivial de la suite exacte \eqref{SEextP} pour les triplets $(G,B,T)$ et $(\Ga,\Ba,T)$ en utilisant les isomorphismes \eqref{isoHOrd} et \eqref{isoHOrda}, l'injection verticale de gauche est induite par le foncteur $\IndQp[\Pa^-]$ et l'injection verticale de droite est l'injection naturelle correspondant à $\alpha \in \Delta$.
\end{prop}

On a un produit semi-direct $N = \Na'' \ltimes \Na$.
On fixe un sous-groupe ouvert compact standard $N_0$ de $N(\Qp)$ compatible avec la décomposition radicielle (voir \cite[Appendice A]{JH}). On note $\Naz$ et $\Naz''$ les intersections respectives de $\Na(\Qp)$ et $\Na''(\Qp)$ avec $N_0$. D'après \cite[Proposition A.7]{JH}, on a un produit semi-direct $N_0 = \Naz'' \ltimes \Naz$.
On utilise les résultats de \cite[§ 3.2]{JH} pour calculer la cohomologie de $N_0$ et l'action de Hecke de $T^+$ à travers ce dévissage de $N_0$.

\begin{lemm} \label{lemm:compat0}
On a une injection naturelle $T^+$-équivariante
\begin{equation*}
\Hc[0]\left(\Naz'',\Inda U\right) \hookrightarrow \Hc[0]\left(N_0,\IndQp U\right)
\end{equation*}
dont le localisé en $T^+$, à travers les isomorphismes \eqref{isoOrd} et \eqref{isoOrda}, est l'identité sur $U$.
\end{lemm}

\begin{proof}
En utilisant le lemme \ref{lemm:inj} pour le triplet $(G,\Pa,\Ga)$ et l'isomorphisme de foncteurs $\IndQp \cong \IndQp[\Pa^-] \Inda$, on voit que l'on a une injection naturelle $T^+ \ltimes \Naz''$-équivariante
\begin{equation} \label{inj}
\Inda U \hookrightarrow \left( \IndQp U \right)^{\Naz}
\end{equation}
et en prenant les invariants par $\Naz''$, on obtient l'injection naturelle $T^+$-équivariante de l'énoncé. Elle s'insère dans un diagramme commutatif de représentations lisses de $T^+$ sur $A$
\begin{equation*} \begin{tikzcd}
\left( \Inda U \right)^{\Naz''} \rar[hook] & \left( \IndQp U \right)^{N_0} \\
U \uar[hook] \rar[equals] & U \uar[hook]
\end{tikzcd} \end{equation*}
où les injections verticales sont données par le lemme \ref{lemm:inj} pour les triplets $(G,B,T)$ et $(\Ga,\Ba,T)$. En utilisant ce lemme, on en déduit que l'injection construite vérifie bien la propriété de l'énoncé.
\end{proof}

\begin{lemm} \label{lemm:compat1}
On a un morphisme naturel $T^+$-équivariant
\begin{equation*}
\Hc[1]\left(\Naz'',\Inda V\right) \to \Hc[1]\left(N_0,\IndQp V\right)
\end{equation*}
dont le localisé en $T^+$, à travers les isomorphismes \eqref{isoHOrd} et \eqref{isoHOrda}, est l'injection naturelle
\begin{equation*}
V^\alpha \otimes (\oma) \hookrightarrow \bigoplus_{\beta \in \Delta} V^\beta \otimes (\omb).
\end{equation*}
\end{lemm}

\begin{rema}
On peut montrer que le morphisme de l'énoncé est injectif.
\end{rema}

\begin{proof}
Soit $I_0 \subset \IndQp V$ (resp. $I^\alpha_0 \subset \Inda V$) la sous-$B(\Qp)$-représentation (resp. sous-$\Ba(\Qp)$-représentation) constituée des fonctions à support dans l'ouvert $B^-(\Qp)N(\Qp)$ (resp. $\Ba^-(\Qp)\Na''(\Qp)$).
On a un diagramme commutatif de représentations lisses de $T^+ \ltimes \Naz''$ sur $A$
\begin{equation*} \begin{tikzcd}
\Inda V \rar[hook] \dar[two heads] & \left( \IndQp V \right)^{\Naz} \dar[two heads] \\
\left( \Inda V \right) / I^\alpha_0 \rar[hook] & \left( \left( \IndQp V \right) / I_0 \right)^{\Naz} \\
V^\alpha \rar[hook] \uar{\rotatebox{90}{$\sim$}} & \Clis_c(\Na(\Qp),V)^{\Naz} \uar[hook]
\end{tikzcd} \end{equation*}
où le morphisme horizontal supérieur est l'injection \eqref{inj} avec $V$ au lieu de $U$, le morphisme horizontal du milieu est induit par le précédent, le morphisme horizontal inférieur est défini par $v \mapsto 1_{\Naz} v$ avec $1_{\Naz}$ la fonction caractéristique de $\Naz$ sur $\Na(\Qp)$, les morphismes verticaux supérieurs sont les applications quotients et les morphismes verticaux inférieurs sont expliqués dans \cite[§ 2.1]{JH}.
On précise l'action de $T^\times \ltimes \Naz''$ sur les termes inférieurs (pour celui de droite, il faut prendre l'action de Hecke de $T^+$ correspondante et l'action induite de $\Naz''$) : $\Naz''$ agit trivialement sur $V^\alpha$ et par translation à droite sur $\Clis_c(\Na(\Qp),V)$ ; l'action de $T^+$ sur $V^\alpha$ est la restriction de celle de $T(\Qp)$ et l'action de $t \in T^+$ sur $f \in \Clis_c(\Na(\Qp),V)$ est donnée par
\begin{equation*}
(t \cdot f)(n) = (\dot{s}_\alpha t \dot{s}_\alpha^{-1}) \cdot f(t^{-1}nt)
\end{equation*}
pour tout $n \in \Na(\Qp)$.
Enfin, les localisés en $T^+$ des injections horizontales sont des isomorphismes (pour celle du haut cela résulte du lemme \ref{lemm:inj} et on en déduit le résultat pour celle du milieu, tandis que pour celle du bas cela résulte de \cite[§ 3.3]{JH}).

En appliquant le foncteur $\Hc[1](\Naz'',-)$ et en utilisant le morphisme d'inflation $\Hc[1](\Naz'',(-)^{\Naz}) \hookrightarrow \Hc[1](N_0,-)$, on obtient un diagramme commutatif de représentations lisses de $T^+$ sur $A$
\begin{equation*} \begin{tikzcd}
\Hc[1] \left( \Naz'', \Inda V \right) \rar \dar[two heads] & \Hc[1] \left( N_0, \IndQp V \right) \dar[two heads] \\
\Hc[1] \left( \Naz'', \left( \Inda V \right) / I^\alpha_0 \right) \rar & \Hc[1] \left( N_0, \left( \IndQp V \right) / I_0 \right) \\
\Hc[1](\Naz'',V^\alpha) \rar \uar{\rotatebox{90}{$\sim$}} & \Hc[1](N_0,\Clis_c(\Na(\Qp),V)) \uar[hook]
\end{tikzcd} \end{equation*}
où les morphismes verticaux supérieurs sont encore surjectifs et le morphisme vertical inférieur droit est encore injectif (voir \cite[§ 2.2]{JH}) et les localisés en $T^+$ des morphismes horizontaux sont encore des isomorphismes.
De plus, les localisés en $T^+$ des morphismes verticaux supérieurs sont des isomorphismes (car les localisés en $T^+$ de $\Hc[1](\Naz'',I^\alpha_0)$ et $\Hc[1](N_0,I_0)$ sont nuls, voir \cite[§ 3.3]{JH}).

D'après \cite[Proposition 3.1.8]{JH} avec $n=1$, $N_0=\Naz''$ et $V^\alpha$ au lieu de $V$, on a un isomorphisme naturel $T^+$-équivariant
\begin{equation} \label{isoVa}
\Hc[1](\Naz'',V^\alpha) \iso V^\alpha \otimes (\oma).
\end{equation}
L'isomorphisme \eqref{isoHOrda} est la composée des localisés en $T^+$ de l'isomorphisme \eqref{isoVa} et des morphismes verticaux de gauche du précédent diagramme.
De plus, la composée des localisés en $T^+$ de l'isomorphisme \eqref{isoVa} et du morphisme horizontal inférieur et des morphismes verticaux de droite du précédent diagramme est une injection naturelle $T^+$-équivariante
\begin{equation*}
V^\alpha \otimes (\oma) \hookrightarrow \HOrdQp \left( \IndQp V \right)
\end{equation*}
dont la composée avec l'isomorphisme \eqref{isoHOrd} est l'injection de l'énoncé.
On en conclut que le morphisme horizontal supérieur du précédent diagramme vérifie bien la propriété de l'énoncé.
\end{proof}

\begin{proof}[Démonstration de la proposition \ref{prop:compat}]
On remarque que le morphisme vertical de gauche du diagramme de l'énoncé est bien injectif car le foncteur exact $\IndQp[\Pa^-]$ admet un quasi-inverse à gauche d'après le lemme \ref{lemm:inj} pour le triplet $(G,\Pa,\Ga)$.

Soit $\Ea$ une extension entre représentations lisses de $\Ga(\Qp)$ sur $A$
\begin{equation*}
0 \to \Inda V \to \Ea \to \Inda U \to 0.
\end{equation*}
En appliquant le foncteur exact $\IndQp[\Pa^-]$, on obtient une extension entre représentations lisses de $G(\Qp)$ sur $A$
\begin{equation} \label{ext}
0 \to \IndQp V \to \IndQp[\Pa^-] \Ea \to \IndQp U \to 0.
\end{equation}
En prenant les invariants par $\Naz$, on obtient une suite exacte de représentations lisses de $T^+ \ltimes \Naz''$ sur $A$
\begin{equation*}
0 \to \left( \IndQp V \right)^{\Naz} \to \left( \IndQp[\Pa^-] \Ea \right)^{\Naz} \to \left( \IndQp U \right)^{\Naz}
\end{equation*}
et on note $I \subset (\IndQp U)^{\Naz}$ l'image du dernier morphisme.
En utilisant le lemme \ref{lemm:inj} pour le triplet $(G,\Pa,\Ga)$, on en déduit un diagramme commutatif de représentations lisses de $T^+ \ltimes \Naz''$ sur $A$
\begin{equation*} \begin{tikzcd}[column sep=small]
0 \rar & \left( \IndQp V \right)^{\Naz} \rar & \left( \IndQp[\Pa^-] \Ea \right)^{\Naz} \rar & I \rar & 0 \\
0 \rar & \Inda V \uar[hook] \rar & \Ea \uar[hook] \rar & \Inda U \uar[hook] \rar & 0
\end{tikzcd} \end{equation*}
dont les lignes sont exactes. En prenant la cohomologie de $\Naz''$ à valeurs dans ces suites exactes courtes, on obtient un diagramme commutatif de représentations lisses de $T^+$ sur $A$
\begin{equation*} \begin{tikzcd}
\Hc[0] \left( \Naz'',I \right) \rar{\delta'_\alpha} & \Hc[1] \left( \Naz'',\left( \IndQp V \right)^{\Naz} \right) \\
\Hc[0] \left( \Naz'',\Inda U \right) \uar[hook] \rar{\delta_\alpha} & \Hc[1] \left( \Naz'',\Inda V \right) \uar.
\end{tikzcd} \end{equation*}
Par ailleurs, on a un diagramme commutatif de représentations lisses de $T^+$ sur $A$
\begin{equation*} \begin{tikzcd}
\Hc[0] \left( N_0,\IndQp U \right) \rar{\delta} & \Hc[1] \left( N_0,\IndQp V \right) \\
\Hc[0] \left( \Naz'',I \right) \uar[hook] \rar{\delta'_\alpha} & \Hc[1] \left( \Naz'',\left( \IndQp V \right)^{\Naz} \right) \uar[hook]
\end{tikzcd} \end{equation*}
où l'injection verticale de gauche est induite par l'inclusion $I \subset \IndQp U$, le morphisme vertical de droite est l'inflation et $\delta$ est le morphisme obtenu par la suite exacte longue de cohomologie associée à la suite exacte \eqref{ext}.

En combinant les deux précédents diagrammes, on obtient un diagramme commutatif de représentations lisses de $T^+$ sur $A$
\begin{equation*} \begin{tikzcd}
\Hc[0] \left( N_0, \IndQp U \right) \rar{\delta} & \Hc[1] \left( N_0, \IndQp V \right) \\
\Hc[0] \left( \Naz'',\Inda U \right) \uar[hook] \rar{\delta_\alpha} & \Hc[1] \left( \Naz'',\Inda V \right) \uar
\end{tikzcd} \end{equation*}
où l'injection verticale de gauche est celle du lemme \ref{lemm:compat0} et le morphisme vertical de droite est celui du lemme \ref{lemm:compat1}. En utilisant ces deux lemmes, on voit qu'à travers les isomorphismes \eqref{isoOrd}, \eqref{isoOrda}, \eqref{isoHOrd} et \eqref{isoHOrda}, le localisé en $T^+$ de ce diagramme est un diagramme commutatif de représentations lisses de $T(\Qp)$ sur $A$
\begin{equation*} \begin{tikzcd}
U \rar{\widetilde{\delta}} & \bigoplus_{\beta \in \Delta} V^\beta \otimes (\omb) \\
U \uar[equals] \rar{\widetilde{\delta}_\alpha} & V^\alpha \otimes (\oma) \uar[hook]
\end{tikzcd} \end{equation*}
où le morphisme vertical de droite est l'injection naturelle. Enfin, $\widetilde{\delta}$ (resp. $\widetilde{\delta}_\alpha$) est l'image de la classe de l'extension $\IndQp[\Pa^-] \Ea$ (resp. $\Ea$) par le morphisme horizontal supérieur (resp. inférieur) du diagramme de l'énoncé.
\end{proof}

\subsection{Calculs sur les représentations ordinaires} \label{ssec:ord}

Soient $P \subset G$ un sous-groupe parabolique standard et $L \subset P$ le sous-groupe de Levi standard.
Pour tout sous-groupe parabolique standard $Q \subset L$, on note $L_Q \subset Q$ le sous-groupe de Levi standard et $Q^- \subset L$ le sous-groupe parabolique opposé à $Q$ par rapport à $L_Q$. Dans ce cas, $BQ \subset P$ est un sous-groupe parabolique standard de $G$, $L_Q \subset BQ$ est le sous-groupe de Levi standard, on note $W_{L_Q} \subset W$ le groupe de Weyl de $(L_Q,T)$ et on pose
\begin{equation*}
\W_{BQ} \dfn \left\{ \text{$w \in W$ | $w$ de longueur maximale dans $W_{L_Q} w$} \right\}.
\end{equation*}
Pour tout sous-groupe parabolique standard $Q \subset Q' \subset L$, on a $\W_{BQ'} \subset \W_{BQ}$. De plus, $1 \in \W_{BQ}$ si et seulement si $Q=B_L$ ; $s_\alpha \in \W_{BQ}$ avec $\alpha \in \Delta-\Delta_L$ si et seulement si $Q=B_L$ ; $s_\alpha \in \W_{BQ}$ avec $\alpha \in \Delta_L$ si et seulement si $Q=Q_\alpha$ avec $Q_\alpha \subset L$ le sous-groupe parabolique standard correspondant à $\alpha$.

\begin{defi}
Pour tout sous-groupe parabolique standard $Q \subset L$, on définit la \emph{représentation spéciale} relative à $Q$ de $L(F)$ sur $A$
\begin{equation*}
\Sp_Q \dfn \frac{\IndL[Q^-] \un}{\sum_{Q \subsetneqq Q' \subset L} \IndL[Q'^-] \un}
\end{equation*}
avec $Q'$ parmi les sous-groupes de $L$ et $\un$ la représentation triviale sur $A$. Lorsque $Q=B_L$, on l'appelle la \emph{représentation de Steinberg} de $L(F)$ sur $A$ et on la note $\St$. Lorsque $Q=Q_\alpha$ avec $\alpha \in \Delta_L$, on la note $\Sp_\alpha$.
\end{defi}

\begin{rema} \label{rema:spe}
Les représentations spéciales sont des représentations lisses admissibles de $L(F)$ sur $A$.
Lorsque $A=\ke$, elles sont irréductibles (voir \cite[Théorème 2]{Vig} pour $\St$, \cite[Corollaire 4.3]{GK} lorsque le système de racines de $L$ ne contient pas de facteur exceptionnel et \cite[Théorème 7.2]{Her} dans le cas général déployé) et deux à deux non isomorphes d'après \cite[Corollaire 4.4 (a)]{GK}.
Dans ce cas, elles forment les constituants irréductibles de $\IndL \un$, chacune apparaissant avec multiplicité un (voir \cite[Corollaire 4.4 (b)]{GK} lorsque le système de racines de $L$ ne contient pas de facteur exceptionnel et \cite[Théorème 7.3]{Her} dans le cas général déployé).
\end{rema}

Soit $\sigma$ une représentation spéciale de $L(F)$ sur $A$. On note $Q \subset L$ le sous-groupe parabolique standard correspondant et on pose
\begin{equation*}
\W_\sigma \dfn \W_{BQ} - \bigcup_{Q \subsetneqq Q' \subset L} \W_{BQ'}
\end{equation*}
avec $Q'$ parmi les sous-groupes paraboliques de $L$.
Par convention, $\w_\sigma$ désigne toujours un élément de $\W_\sigma$.

Soit $\chi : L(F) \to A^\times$ un caractère lisse. On rappelle que la restriction à $T(F)$ identifie les caractères de $L(F)$ avec les caractères de $T(F)$ triviaux sur $\im \alpha^\vee$ pour tout $\alpha \in \Delta_L$ (voir \cite[Proposition 3.3]{Abe}).

\begin{defi}
On appelle \emph{représentation ordinaire}\footnote{Cette terminologie, empruntée à \cite{Pas}, est justifiée par le fait suivant : lorsque $A=\ke$, les représentations ordinaires irréductibles sont exactement les sous-quotients irréductibles de séries principales. \label{foot:ord}} de $G(F)$ sur $A$ une représentation de la forme $\Ind[P^-] (\sigma \otimes \chi)$.
\end{defi}

On commence par définir la \emph{filtration de Bruhat} de $\Ind[P^-] (\sigma \otimes \chi)$. On note $d$ la dimension de $N$ et pour tout $w \in W$, on note $N_w \subset N$ le sous-groupe fermé $N \cap \dot{w}^{-1} N \dot{w}$.

\begin{prop} \label{prop:filind}
Il existe une filtration $(I_r^\sigma)_{r \in \llbrack -1,d \rrbrack}$ de $\Ind[P^-] (\sigma \otimes \chi)$ par des sous-$B(F)$-représentations et pour tout $r \in \llbrack 0,d \rrbrack$, on a une suite exacte courte de représentations lisses de $B(F)$ sur $A$
\begin{equation*}
0 \to I^\sigma_{r-1} \to I^\sigma_r \to \bigoplus_{\ell(\w_\sigma)=r} \Clis_c(N_{\w_\sigma}(F),\chi) \to 0.
\end{equation*}
\end{prop}

\begin{proof}
D'après \cite[§ 2.3]{JH}\footnote{Ces résultats, énoncés pour les caractères de la forme $\eta \circ \detfr$ avec $\eta : F^\times \to A^\times$ un caractère lisse et $\operatorname{d\acute{e}t} : G \to \GL_1$ un caractère algébrique, sont encore vrais pour un caractère lisse $\chi : L(F) \to A^\times$ quelconque. \label{foot:car}}, il existe pour tout sous-groupe parabolique standard $Q' \subset L$ une filtration $(I_r^{BQ'})_{r \in \llbrack -1,d \rrbrack}$ de $\Ind[(BQ')^-] \chi$ et pour tout $r \in \llbrack 0,d \rrbrack$, on a une suite exacte courte de représentations lisses de $B(F)$ sur $A$
\begin{equation} \label{filBQ}
0 \to I^{BQ'}_{r-1} \to I^{BQ'}_r \to \bigoplus_{\ell(\w_{BQ'})=r} \Clis_c(N_{\w_{BQ'}}(F),\chi) \to 0.
\end{equation}
De plus pour tous sous-groupes paraboliques standards $Q'_1 \subset Q'_2 \subset L$, l'injection $G(F)$-équivariante
\begin{equation} \label{injQ12}
\Ind[(BQ'_2)^-] \chi \hookrightarrow \Ind[(BQ'_1)^-] \chi
\end{equation}
est stricte par rapport aux filtrations $(I_r^{BQ'_1})_{r \in \llbrack -1,d \rrbrack}$ et $(I_r^{BQ'_2})_{r \in \llbrack -1,d \rrbrack}$ : pour tout $r \in \llbrack 0,d \rrbrack$, on a un diagramme commutatif de représentations lisses de $B(F)$ sur $A$
\begin{equation} \label{filBQ12} \begin{tikzcd}
0 \rar & I^{BQ'_1}_{r-1} \rar & I^{BQ'_1}_r \rar & \displaystyle \bigoplus_{\ell(\w_{BQ'_1})=r} \Clis_c(N_{\w_{BQ'_1}}(F),\chi) \rar & 0 \\
0 \rar & I^{BQ'_2}_{r-1} \uar[hook] \rar & I^{BQ'_2}_r \uar[hook] \rar & \displaystyle \bigoplus_{\ell(\w_{BQ'_2})=r} \Clis_c(N_{\w_{BQ'_2}}(F),\chi) \uar[hook] \rar & 0
\end{tikzcd} \end{equation}
où les injections verticales de gauche et du milieu sont induites par l'injection \eqref{injQ12} et l'injection verticale de droite est l'injection naturelle correspondant à l'inclusion $\W_{BQ'_2} \subset \W_{BQ'_1}$.
Enfin, on a un isomorphisme $G(F)$-équivariant
\begin{equation} \label{isoIndSp}
\Ind[P^-] (\sigma \otimes \chi) \cong \frac{\Ind[(BQ)^-] \chi}{\sum_{Q \subsetneqq Q' \subset L} \Ind[(BQ')^-] \chi}
\end{equation}
avec $Q'$ parmi les sous-groupes paraboliques de $L$, qui résulte de l'exactitude de la torsion par $\chi$ et du foncteur $\Ind[P^-]$, ainsi que de l'isomorphisme $L(F)$-équivariant $(\IndL[Q'^-] \un) \otimes \chi \cong \IndL[Q'^-] \chi$.
Soit $(I^\sigma_r)_{r \in \llbrack -1,d \rrbrack}$ l'image de la filtration $(I_r^{BQ})_{r \in \llbrack -1,d \rrbrack}$ de $\Ind[(BQ)^-] \chi$ dans le quotient $\Ind[P^-] (\sigma \otimes \chi)$.
Un résultat de \cite{AHV} (qui se démontre de la même façon que \cite[V.16 Lemme 9]{AHHV}) montre que $I_{r-1}^{BQ} \cap (\sum_{Q \subsetneqq Q' \subset L} I_r^{BQ'}) = \sum_{Q \subsetneqq Q' \subset L} I_{r-1}^{BQ'}$ pour tout $r \in \llbrack 0,d \rrbrack$, donc l'isomorphisme \eqref{isoIndSp} induit un isomorphisme $B(F)$-équivariant
\begin{equation} \label{isofilIndSp}
I^\sigma_r/I^\sigma_{r-1} \cong \frac{I^{BQ}_r/I^{BQ}_{r-1}}{\sum_{Q \subsetneqq Q' \subset L} (I^{BQ'}_r/I^{BQ'}_{r-1})}.
\end{equation}
En utilisant le diagramme \eqref{filBQ12} avec $Q_1=Q$ et $Q_2=Q'$ pour tout sous-groupe parabolique $Q \subsetneqq Q' \subset L$, on déduit de l'isomorphisme \eqref{isofilIndSp} que la filtration $(I_r^\sigma)_{r \in \llbrack -1,d \rrbrack}$ vérifie bien la propriété de l'énoncé.
\end{proof}

\begin{rema}
Cette filtration est déjà définie et étudiée dans \cite{Vig} pour les représentations spéciales de $G(F)$ (c'est-à-dire lorsque $P=G$).
\end{rema}

Soit $N_0$ un sous-groupe ouvert compact de $N(F)$. On montre que la filtration de Bruhat de $\Ind[P^-] (\sigma \otimes \chi)$ induit une filtration des $A$-modules $\Hc(N_0,\Ind[P^-] (\sigma \otimes \chi))$.

\begin{prop} \label{prop:filcohom}
Pour tout $r \in \llbrack 0,d \rrbrack$, on a des suites exactes courtes de représentations lisses de $T^+$ sur $A$
\begin{equation*}
0 \to \Hc(N_0,I_{r-1}^\sigma) \to \Hc(N_0,I_r^\sigma) \to \bigoplus_{\ell(\w_\sigma)=r} \Hc(N_0,\Clis_c(N_{\w_\sigma}(F),\chi)) \to 0.
\end{equation*}
\end{prop}

\begin{proof}
On reprend les notations de la preuve de la proposition \ref{prop:filind} et on fixe $r \in \llbrack 0,d \rrbrack$.
On a des diagrammes commutatifs de représentations lisses de $T^+$ sur $A$
\begin{equation*} \begin{tikzcd}[column sep=tiny,row sep=small]
0 \rar & \Hc(N_0,I^{BQ}_{r-1}) \dar[two heads] \rar & \Hc(N_0,I^{BQ}_r) \dar[two heads] \rar & \displaystyle \bigoplus_{\ell(\w_{BQ})=r} \Hc(N_0,\Clis_c(N_{\w_{BQ}}(F),\chi)) \dar[two heads] \rar & 0 \\
0 \rar & \Hc(N_0,I^\sigma_{r-1}) \rar & \Hc(N_0,I^\sigma_r) \rar & \displaystyle \bigoplus_{\ell(\w_\sigma)=r} \Hc(N_0,\Clis_c(N_{\w_\sigma}(F),\chi)) \rar & 0
\end{tikzcd} \end{equation*}
dont les lignes supérieures sont exactes d'après \cite[§ 2.3]{JH} (voir la note de bas de page \ref{foot:car}).
La surjectivité des morphismes verticaux de droite est immédiate, ce qui permet de montrer que le dernier morphisme non trivial de chaque ligne inférieure est surjectif.
Par la suite exacte longue de cohomologie associée à la suite exacte courte de la proposition \ref{prop:filind}, on en conclut que les lignes inférieures sont exactes.
\end{proof}

On calcule maintenant les parties ordinaires dérivées d'une représentation ordinaire de $G(F)$ sur $A$. Pour tout $w \in W$, on note $\alpha_w$ le caractère algébrique de la représentation adjointe de $T$ sur $\detfr_F \Lie(N_{w_0w})$ avec $w_0 \in W$ l'élément de longueur maximale. On a $\alpha_1=0$ et $\alpha_{s_\alpha}=\alpha$ pour tout $\alpha \in \Delta$.

\begin{theo} \label{theo:hord}
Soient $\sigma$ une représentation spéciale de $L(F)$ sur $A$, $\chi : L(F) \to A^\times$ un caractère lisse et $n \in \Nbb$. On suppose $A=\ke$ ou $n \leq 1$. Alors on a un isomorphisme $T(F)$-équivariant
\begin{equation*}
\HOrd[n] \left( \Ind[P^-] (\sigma \otimes \chi) \right) \cong \bigoplus_{[F:\Qp] \cdot \ell(\w_\sigma) = n} \w_\sigma^{-1}(\chi) \cdot (\oma_{\w_\sigma}).
\end{equation*}
\end{theo}

\begin{proof}
D'après \cite[Théorème 3.4.7]{Em2}, les termes des suites exactes de la proposition \ref{prop:filcohom} sont réunions de sous-$A$-modules de type fini stables par $T^+$. En utilisant \cite[Lemme 3.2.1]{Em2}, on en déduit des suites exactes courtes de représentations lisses de $T(F)$ sur $A$
\begin{multline*}
0 \to \HOrd(I_{r-1}^\sigma) \to \HOrd(I_r^\sigma) \\
\to \bigoplus_{\ell(\w_\sigma)=r} \HOrd(\Clis_c(N_{\w_\sigma}(F),\chi)) \to 0
\end{multline*}
pour tout $r \in \llbrack 0,d \rrbrack$.
On conclut en utilisant \cite[Théorème 4.2.2]{JH}.
\end{proof}

On explicite enfin ce calcul en degrés $0$ et $1$.

\begin{coro} \label{coro:ord}
Soient $\sigma$ une représentation spéciale de $L(F)$ sur $A$ et $\chi : L(F) \to A^\times$ un caractère lisse. On a un isomorphisme $T(F)$-équivariant
\begin{equation*}
\Ord \left( \Ind[P^-] (\sigma \otimes \chi) \right)
\cong
\begin{cases}
\chi & \text{si $\sigma=\St$,} \\
0 & \text{sinon}.
\end{cases}
\end{equation*}
\end{coro}

\begin{proof}
Soit $Q \subset L$ le sous-groupe parabolique correspondant à $\sigma$. On utilise le théorème \ref{theo:hord} avec $n=0$. La somme directe de l'isomorphisme est nulle sauf si $1 \in \W_\sigma$, auquel cas $1 \in \W_{BQ}$ (car $\W_\sigma \subset \W_{BQ}$) donc $Q=B_L$ d'où $\sigma=\St$ et on obtient ainsi l'isomorphisme de l'énoncé car $1$ est l’unique élément de longueur $0$ de $\W_\St$.
\end{proof}

\begin{coro} \label{coro:h1ord}
Soient $\sigma$ une représentation spéciale de $L(F)$ sur $A$ et $\chi : L(F) \to A^\times$ un caractère lisse.
\begin{enumerate}[(i)]
\item Si $F=\Qp$, alors on a un isomorphisme $T(\Qp)$-équivariant
\begin{multline*}
\HOrdQp \left( \IndQp[P^-] (\sigma \otimes \chi) \right) \\
\cong
\begin{cases}
\bigoplus_{\alpha \in \Delta - \Delta_L} s_\alpha(\chi) \cdot (\oma) & \text{si $\sigma=\St$,} \\
\chi \cdot (\oma) & \text{si $\sigma = \Sp_\alpha$ avec $\alpha \in \Delta_L$,} \\
0 & \text{sinon}.
\end{cases}
\end{multline*}
\item Si $F \neq \Qp$, alors $\HOrd[1] (\Ind[P^-] (\sigma \otimes \chi)) = 0$.
\end{enumerate}
\end{coro}

\begin{proof}
Soit $Q \subset L$ le sous-groupe parabolique correspondant à $\sigma$. On utilise le théorème \ref{theo:hord} avec $n=1$. La somme directe de l'isomorphisme est nulle sauf si $F=\Qp$ et s'il existe $\alpha \in \Delta$ tel que $s_\alpha \in \W_\sigma$. Dans ce cas $s_\alpha \in \W_{BQ}$ (car $\W_\sigma \subset \W_{BQ}$) et ou bien $\alpha \in \Delta-\Delta_L$ donc $Q=B_L$ d'où $\sigma=\St$, ou bien $\alpha \in \Delta_L$ donc $Q=Q_\alpha$ d'où $\sigma=\Sp_\alpha$.
On obtient les isomorphismes de l'énoncé car d'une part les éléments de longueur $1$ de $\W_{\St}$ sont exactement les $s_\alpha$ avec $\alpha \in \Delta-\Delta_L$ et d'autre part si $\alpha \in \Delta_L$, alors $s_\alpha$ est l'unique élément de longueur $1$ de $\W_{\Sp_\alpha}$ et $s_\alpha(\chi)=\chi$.
\end{proof}

\section{Extensions entre induites}

Nous montrons l'existence de certaines extensions entre séries principales d'un groupe réductif de rang semi-simple $1$. Nous calculons ensuite les extensions entre certaines induites de $G(F)$ dans la catégorie des représentations continues unitaires admissibles de $G(F)$ sur $E$ (en utilisant les résultats pour les groupes réductifs de rang semi-simple $1$).
Nos démonstrations prouvent également les résultats analogues modulo $p$ (c'est-à-dire dans la catégorie des représentations lisses admissibles de $G(F)$ sur $\ke$).

\subsection{\texorpdfstring{Groupes réductifs de rang semi-simple $1$}{Groupes réductifs de rang semi-simple 1}} \label{ssec:ga}

On suppose $F=\Qp$ et on fixe $\alpha \in \Delta$. On reprend les notations de la sous-section \ref{ssec:compat}.
On commence par donner une description explicite de $\Ga$.

\begin{lemm} \label{lemm:Ga}
On a un isomorphisme $\Ga \cong T' \times \Ga'$ avec $T' \subset T$ un sous-tore et $\Ga' \in \{\SL_2,\GL_2,\PGL_2\}$.
\end{lemm}

\begin{proof}
On note $\Za^\circ$ la composante neutre de $\Za$ et $\Ga^\der$ le groupe dérivé de $\Ga$. On utilise la suite exacte courte naturelle (voir \cite[Partie II, § 1.18]{Jan})
\begin{equation*}
1 \to \Za^\circ \cap \Ga^\der \to \Za^\circ \times \Ga^\der \to \Ga \to 1.
\end{equation*}
Comme $\Ga^\der$ est semi-simple de rang $1$, il est isomorphe à $\PGL_2$ ou $\SL_2$.
Si $\Ga^\der=\PGL_2$, alors $\Za^\circ \cap \Ga^\der = 1$, d'où $\Ga \cong \Za^\circ \times \PGL_2$.
On suppose $\Ga^\der=\SL_2$, donc $\Za^\circ \cap \Ga^\der$ est isomorphe à $1$ ou $\mu_2$.
Si $\Za^\circ \cap \Ga^\der=1$, alors $\Ga \cong \Za^\circ \times \SL_2$.
Sinon $\Za^\circ \cap \Ga^\der=\mu_2$ et comme $\Za^\circ$ est un tore, il existe un sous-tore $T' \subset \Za^\circ$ et un isomorphisme $\Za^\circ \cong T' \times \GL_1$ à travers lequel l'inclusion $\Za^\circ \cap \Ga^\der \subset \Za^\circ$ s'identifie à la composée $\mu_2 \hookrightarrow \GL_1 \hookrightarrow T' \times \GL_1$, d'où un isomorphisme $\Ga \cong T' \times (\GL_1 \times \SL_2)/\mu_2 \cong T' \times \GL_2$.
\end{proof}

On déduit du lemme \ref{lemm:Ga} le résultat suivant, où l'on identifie certains caractères \og irréguliers \fg.

\begin{lemm} \label{lemm:irreg}
Soient $\chi$ et $\eta$ des caractères de $T(\Qp)$ et $\Qp^\times$ respectivement à valeurs dans le groupe des unités d'un anneau quelconque.
Si $s_\alpha(\chi) \cdot (\eta \circ \alpha) \neq \chi$, alors $\chi \circ \alpha^\vee \neq \eta$.
La réciproque est vraie sauf dans le cas suivant : $\Ga'=\SL_2$ et $\chi \circ \alpha^\vee = \epsilon \eta$ avec $\epsilon : \Qp^\times \to \{\pm1\}$ un morphisme de groupes non trivial.
\end{lemm}

\begin{rema} \label{rema:irregmod2}
Il n'existe pas de caractère irrégulier modulo $2$ : si $p=2$ et $\chi : T(\Qp) \to \ke^\times$ est un caractère lisse, alors $s_\alpha(\chi)=\chi$ si et seulement si $\chi \circ \alpha^\vee = 1$.
\end{rema}

À présent, on prouve l'existence d'extensions non scindées entre certaines séries principales continues unitaires de $\Ga(\Qp)$ sur $E$. Le résultat analogue modulo $p$ se démontre de façon analogue.
Le caractère de $\Za(\Qp)$ en indice des $\Ext^1$ signifie que l'on se restreint aux catégories de représentations sur lesquelles $\Za(\Qp)$ agit à travers ce caractère.

\begin{lemm} \label{lemm:exta}
Soit $\chi : T(\Qp) \to \Oe^\times \subset E^\times$ un caractère continu unitaire. Si $\chi \circ \alpha^\vee \neq \varepsilon^{-1}$, alors l'injection $E$-linéaire
\begin{multline*}
\Ext^1_{T(\Qp),\chi_{|\Za(\Qp)}} \left( s_\alpha(\chi) \cdot (\epsa),\chi \right) \\
\hookrightarrow \Ext^1_{\Ga(\Qp),\chi_{|\Za(\Qp)}} \left( \Inda s_\alpha(\chi) \cdot (\epsa),\Inda \chi \right)
\end{multline*}
induite par le foncteur $\Inda$ n'est pas surjective.
\end{lemm}

\begin{proof}
On suppose $\chi \circ \alpha^\vee \neq \varepsilon^{-1}$.
En utilisant le lemme \ref{lemm:Ga} et en notant $\chi'$ et $\chi'_\alpha$ les restrictions de $\chi$ à $T'(\Qp)$ et $(T \cap \Ga')(\Qp)$ respectivement, on voit que le produit tensoriel avec $\chi'$ sur $E$ induit des isomorphismes $E$-linéaires
\begin{multline*}
\Ext^1_{(T \cap \Ga')(\Qp),\chi_{|(\Za \cap \Ga')(\Qp)}} \left( s_\alpha(\chi'_\alpha) \cdot (\epsa),\chi'_\alpha \right) \\
\iso \Ext^1_{T(\Qp),\chi_{|\Za(\Qp)}} \left( s_\alpha(\chi) \cdot (\epsa),\chi \right)
\end{multline*}
et
\begin{multline*}
\Ext^1_{\Ga'(\Qp),\chi_{|(\Za \cap \Ga')(\Qp)}} \left( \IndQpH s_\alpha(\chi'_\alpha) \cdot (\epsa),\IndQpH \chi'_\alpha \right) \\
\iso \Ext^1_{\Ga(\Qp),\chi_{|\Za(\Qp)}} \left( \Inda s_\alpha(\chi) \cdot (\epsa),\Inda \chi \right)
\end{multline*}
compatibles avec les injections $E$-linéaires induites par les foncteurs $\Inda$ et $\Indaprime$, donc il suffit de montrer le résultat lorsque $\Ga=\Ga'$.

On suppose $\Ga=\GL_2$ et on note $T_2 \subset \GL_2$ le sous-groupe des matrices diagonales, $B_2^- \subset \GL_2$ le sous-groupe des matrices triangulaires inférieures et $Z_2$ le centre de $\GL_2$. On a $s_\alpha(\chi) \cdot (\epsa) \neq \chi$ d'après le lemme \ref{lemm:irreg}, donc $\Ext_{T(\Qp),\chi_{|\Za(\Qp)}}^1(s_\alpha(\chi) \cdot (\epsa),\chi)=0$ d'après \cite[Proposition 5.1.6]{JH}. De plus \cite[Conjecture 3.7.2]{Em2} est vraie (voir \cite{EmP}), ce qui permet en procédant par réduction modulo $\pe^k$ et dévissage comme dans la preuve de \cite[Proposition B.2 (i)]{BH} (mais en se restreignant aux représentations à caractère central, voir \cite[§ 7.1]{Pas}) de montrer l'existence d'une extension non scindée à caractère central
\begin{equation*}
0 \to \IndQpGLd \chi \to \E_2 \to \IndQpGLd s_\alpha(\chi) \cdot (\epsa) \to 0.
\end{equation*}

On suppose $\Ga=\SL_2$ et on prolonge arbitrairement $\chi$ en un caractère continu unitaire de $T_2(\Qp)$. Par ce qui précède, il existe une extension non scindée à caractère central $\E_2$ entre les séries principales de $\GL_2(\Qp)$ correspondantes. Par restriction, on obtient une extension à caractère central de $\IndQpSLd s_\alpha(\chi) \cdot (\epsa)$ par $\IndQpSLd \chi$, qui n'est pas induite à partir d'une extension de $s_\alpha(\chi) \cdot (\epsa)$ par $\chi$ car ses parties ordinaires, dont la restriction à $(T_2 \cap \SL_2)(\Qp)$ est naturellement isomorphe aux parties ordinaires de $\E_2$, sont de dimension $1$ sur $E$.

On suppose $\Ga=\PGL_2$ et on identifie $\chi$ à un caractère de $T_2(\Qp)$ trivial sur $Z_2(\Qp)$. On a $s_\alpha(\chi) \cdot (\epsa) \neq \chi$ d'après le lemme \ref{lemm:irreg}, donc $\Ext^1_{(T_2/Z_2)(\Qp)}(s_\alpha(\chi) \cdot (\epsa),\chi)=0$ d'après \cite[Proposition 5.1.6]{JH}. Par ce qui précède, il existe une extension non scindée à caractère central $\E_2$ entre les séries principales de $\GL_2(\Qp)$ correspondantes. Par inflation, on obtient une extension non scindée de $\IndQpPGLd s_\alpha(\chi) \cdot (\epsa)$ par $\IndQpPGLd \chi$.
\end{proof}

On prouve également l'existence d'une auto-extension supplémentaire d'une série principale lisse réductible de $\Ga(\Qp)$ sur $\ke$ lorsque $p=2$.

\begin{lemm} \label{lemm:extamod2}
Soit $\chi : T(\Qp) \to \ke^\times$ un caractère lisse. Si $\chi \circ \alpha^\vee=1$ et $p=2$, alors l'injection $\ke$-linéaire
\begin{equation*}
\Ext^1_{T(\Qp),\chi_{|\Za(\Qp)}} \left( \chi,\chi \right) \hookrightarrow \Ext^1_{\Ga(\Qp),\chi_{|\Za(\Qp)}} \left( \Inda \chi,\Inda \chi \right)
\end{equation*}
induite par le foncteur $\Inda$ n'est pas surjective.
\end{lemm}

\begin{proof}
En supposant seulement $\chi \circ \alpha^\vee=1$ et en procédant comme dans la preuve du lemme \ref{lemm:exta} (c'est-à-dire en se ramenant au cas $\Ga=\GL_2$, puis en utilisant \cite[Proposition 4.3.13 (2)]{Em2} mais en se restreignant aux représentations à caractère central), on montre que
\begin{equation} \label{extaSpa}
\Ext^1_{\Ga(\Qp),\chi_{|\Za(\Qp)}} \left( \Inda \chi \cdot (\oma), \chi \right) \neq 0.
\end{equation}
On suppose $\chi \circ \alpha^\vee=1$ et $p=2$.
En procédant comme dans la preuve de \cite[Proposition 4.3.15 (4)]{Em2}, on voit que le foncteur $\Inda$ et l'injection $\ke$-linéaire $\chi \hookrightarrow \Inda \chi$ induisent un isomorphisme $\ke$-linéaire
\begin{multline*}
\Ext^1_{T(\Qp),\chi_{|\Za(\Qp)}} \left( \chi,\chi \right) \oplus \Ext^1_{\Ga(\Qp),\chi_{|\Za(\Qp)}} \left( \Inda \chi,\chi \right) \\
\iso \Ext^1_{\Ga(\Qp),\chi_{|\Za(\Qp)}} \left( \Inda \chi,\Inda \chi \right).
\end{multline*}
En utilisant l'inégalité \eqref{extaSpa} avec $\oma=1$, on obtient le résultat.
\end{proof}

Enfin, on prouve un dernier résultat sur certaines auto-extensions modulo $p$ de longueur $2$.
On calcule les $\Ext^2$ dans les catégories de représentations lisses localement admissibles sur $\ke$ sur lesquelles $\Za(\Qp)$ agit à travers le caractère en indice.

\begin{lemm} \label{lemm:exta2}
On suppose $\Za$ connexe et $p \neq 2$. Soit $\chi : T(\Qp) \to \ke^\times$ un caractère lisse. Si $\chi \circ \alpha^\vee = \omega^{-1}$, alors le morphisme $\ke$-linéaire
\begin{equation*}
\Ext^2_{T(\Qp),\chi_{|\Za(\Qp)}} \left( \chi,\chi \right) \to \Ext^2_{\Ga(\Qp),\chi_{|\Za(\Qp)}} \left( \Inda \chi,\Inda \chi \right)
\end{equation*}
induit par le foncteur $\Inda$ n'est pas injectif.
\end{lemm}

\begin{proof}
Comme $\Za$ est connexe, on déduit du lemme \ref{lemm:Ga} que $\Ga$ est le produit direct d'un tore et de $\GL_2$ ou $\PGL_2$.
On en déduit un morphisme $\GL_2 \to \Ga$ qui induit un isomorphisme entre la catégorie des représentations lisses localement admissibles de $\Ga(\Qp)$ sur $\ke$ ayant pour caractère central $\chi_{|\Za(\Qp)}$ et la catégorie des représentations lisses localement admissibles de $\GL_2(\Qp)$ sur $\ke$ ayant pour caractère central la composée du morphisme induit $Z_2(\Qp) \to \Za(\Qp)$ et de $\chi_{|\Za(\Qp)}$, d'où une suite exacte de $\ke$-espaces vectoriels
\begin{multline} \label{SEGrotha}
0 \to \Ext^1_{T(\Qp),\chi_{|\Za(\Qp)}} \left( \chi,\chi \right) \to \Ext^1_{\Ga(\Qp),\chi_{|\Za(\Qp)}} \left( \Inda \chi,\Inda \chi \right) \\
\to \Hom_{T(\Qp)} \left( \chi,s_\alpha(\chi) \cdot (\oma) \right) \to \Ext^2_{T(\Qp),\chi_{|\Za(\Qp)}} \left( \chi,\chi \right) \\
\to \Ext^2_{\Ga(\Qp),\chi_{|\Za(\Qp)}} \left( \Inda \chi,\Inda \chi \right)
\end{multline}
où le premier morphisme non trivial et le dernier morphisme sont induits par le foncteur $\Inda$ (voir \cite[§ 7.1]{Pas}).

On suppose $\chi \circ \alpha^\vee = \omega^{-1}$. Comme $p \neq 2$ on a $\chi \circ \alpha^\vee \neq 1$, donc $\Inda \chi$ est irréductible d'après \cite[Théorème 30 (1) (a)]{BL}. De plus, le premier morphisme non trivial de la suite exacte \eqref{SEGrotha} est un isomorphisme d'après \cite[Proposition 4.3.15 (3)]{Em2} (pour tenir compte du caractère central, on utilise \cite[Proposition 8.1]{PasExt}).
Ainsi, le dernier morphisme de la suite exacte \eqref{SEGrotha} n'est pas injectif.
\end{proof}

\begin{rema} \label{rema:exta2irreg}
Le résultat est en fait vrai sans supposer $\Za$ connexe : pour $\Ga=\GL_2$ le noyau du dernier morphisme de la suite exacte \eqref{SEGrotha} est engendré par le produit de Yoneda de $\overline{\val} \circ \alpha$ et $\overline{\log_p} \circ \alpha$ vues comme des extensions à travers l'isomorphisme $\Ext^1_{T_2(\Qp),\chi_{|Z_2(\Qp)}}(\chi,\chi) \cong \Hc[1]((T_2/Z_2)(\Qp),\un)$ (voir \cite[Appendice A]{Em2}) et les restrictions à $(T_2 \cap \SL_2)(\Qp)$ de ces dernières sont linéairement indépendantes ; on peut en déduire que leur produit de Yoneda est non nul et dans le noyau du dernier morphisme de la suite exacte \eqref{SEGrotha} avec $\Ga=\SL_2$.
\end{rema}

\subsection{Extensions entre séries principales}

On détermine les extensions entre séries principales de $G(F)$.
Lorsque $F \neq \Qp$, ces extensions proviennent toujours d'une extension entre caractères de $T(F)$ (voir \cite[Théorème 1.2]{JH}). On suppose donc $F=\Qp$.

\begin{theo} \label{theo:ext}
Soit $\chi : T(\Qp) \to \Oe^\times \subset E^\times$ un caractère continu unitaire.
\begin{enumerate}[(i)]
\item Si $\chi' : T(\Qp) \to \Oe^\times \subset E^\times$ est un caractère continu unitaire distinct de $\chi$, alors
\begin{multline*}
\dim_E \Ext^1_{G(\Qp)} \left( \IndQp \chi',\IndQp \chi \right) \\
= \card \left\{ \alpha \in \Delta \mid \chi' = s_\alpha(\chi) \cdot (\epsa) \right\}.
\end{multline*}
\item Si $s_\alpha(\chi) \cdot (\epsa) \neq \chi$ pour tout $\alpha \in \Delta$, alors le foncteur $\IndQp$ induit un isomorphisme $E$-linéaire
\begin{equation*}
\Ext_{T(\Qp)}^1 \left( \chi, \chi \right) \iso \Ext_{G(\Qp)}^1 \left( \IndQp \chi,\IndQp \chi \right).
\end{equation*}
\end{enumerate}
\end{theo}

\begin{rema} \phantomsection \label{rema:ext}
\begin{enumerate}[(i)]
\item Lorsque le centre de $G$ est connexe, la dimension du point (i) du théorème est au plus $1$ (par une adaptation de \cite[Lemme 5.1.3]{JH}).
En général, cette dimension peut être arbitrairement grande (voir l'exemple \ref{exem:card} ci-dessous).
\item Sans l'hypothèse de généricité du point (ii), on montre que le foncteur $\IndQp$ induit une injection $E$-linéaire dont le conoyau est de dimension comprise entre
\begin{gather*}
\card \left\{ \alpha \in \Delta \mid s_\alpha(\chi) \cdot (\epsa) = \chi ~ \text{et} ~ \chi \circ \alpha^\vee \neq \varepsilon^{-1} \right\} \\
\text{et} ~ \card \left\{ \alpha \in \Delta \mid s_\alpha(\chi) \cdot (\epsa) = \chi \right\}.
\end{gather*}
On s'attend à ce que la minoration soit une égalité. On le prouve lorsque le centre de $G$ est connexe et $p \neq 2$ (voir le corollaire \ref{coro:autoext} ci-dessous), auquel cas le minorant est nul d'après le lemme \ref{lemm:irreg}.
De même pour son analogue modulo $p$, sauf dans certains cas exceptionnels et lorsque $p=2$ (voir la remarque \ref{rema:autoextmodp} ci-dessous).
\end{enumerate}
\end{rema}

\begin{exem} \label{exem:card}
Avec $n \in \Nbb$, $G \subset \GL_{2n}$ l'intersection des sous-groupes diagonaux par blocs $\GL_2^n$ et $\GL_2^k \times \SL_4 \times \GL_2^{n-2-k}$ pour tout $k \in \llbrack 0,n-2 \rrbrack$, $B \subset G$ le sous-groupe des matrices triangulaires supérieures, $T \subset B$ le sous-groupe des matrices diagonales et $\chi : T(\Qp) \to \Oe^\times \subset E^\times$ le caractère continu unitaire défini par
\begin{equation*}
\chi(\diag(t_1,\dots,t_{2n})) = (-1)^{\val(t_1t_3\dots t_{2n-1})} \varepsilon^{-1}(t_1^{2n-1}t_2^{2n-2}\dots t_{2n-1}),
\end{equation*}
on a $\card \Delta = n$ et $s_\alpha(\chi) \cdot (\epsa)$ est distinct de $\chi$ et indépendant de $\alpha \in \Delta$.
\end{exem}

\begin{proof}
Soient $\chi,\chi' : T(\Qp) \to A^\times$ des caractères lisses. En utilisant les isomorphismes \eqref{isoOrd} et \eqref{isoHOrd}, la suite exacte \eqref{SEextP} pour le triplet $(G,B,T)$ avec $U=\chi'$ et $V = \IndQp \chi$ devient
\begin{multline} \label{SEext}
0 \to \Ext^1_{T(\Qp)} \left( \chi',\chi \right) \to \Ext^1_{G(\Qp)} \left( \IndQp \chi',\IndQp \chi \right) \\
\to \bigoplus_{\alpha \in \Delta} \Hom_{T(\Qp)} \left( \chi',s_\alpha(\chi) \cdot (\oma) \right).
\end{multline}
De même pour tout $\alpha \in \Delta$, en utilisant les isomorphismes \eqref{isoOrda} et \eqref{isoHOrda}, la suite exacte \eqref{SEextP} pour le triplet $(\Ga,\Ba,T)$ avec $U=\chi'$ et $V = \Inda \chi$ devient
\begin{multline} \label{SEexta}
0 \to \Ext^1_{T(\Qp)} \left( \chi',\chi \right) \to \Ext^1_{\Ga(\Qp)} \left( \Inda \chi',\Inda \chi \right) \\
\to \Hom_{T(\Qp)} \left( \chi',s_\alpha(\chi) \cdot (\oma) \right).
\end{multline}

On suppose $A=\ke$ et on prouve la version modulo $p$ du théorème.
On pose
\begin{gather*}
\Delta' \dfn \left\{ \alpha \in \Delta \mid \chi' = s_\alpha(\chi) \cdot (\oma) \right\}, \\
\Delta'' \dfn \left\{ \alpha \in \Delta' \mid \chi \circ \alpha^\vee = \omega^{-1} \right\}.
\end{gather*}
En utilisant la suite exacte \eqref{SEext}, on voit que le conoyau du premier morphisme non trivial est de dimension au plus $\card \Delta'$.
Pour tout $\alpha \in \Delta' - \Delta''$, le premier morphisme non trivial de la suite exacte \eqref{SEexta} n'est pas surjectif d'après l'analogue modulo $p$ du lemme \ref{lemm:exta}, donc il existe une extension
\begin{equation*}
0 \to \Inda \chi \to \Ea \to \Inda \chi' \to 0
\end{equation*}
dont la classe a une image non nulle par le second morphisme non trivial de cette suite exacte.
En utilisant la proposition \ref{prop:compat} avec $U=\chi'$ et $V=\chi$, on en déduit que les images des classes des extensions $(\IndQp[\Pa^-] \Ea)_{\alpha \in \Delta' - \Delta''}$ par le second morphisme non trivial de la suite exacte \eqref{SEext} sont non nulles et appartiennent à des facteurs directs distincts, donc le conoyau du premier morphisme non trivial de la suite exacte \eqref{SEext} est de dimension au moins $\card(\Delta' - \Delta'')$.
Si $\chi' \neq \chi$, alors $\Delta'' = \emptyset$ d'après le lemme \ref{lemm:irreg} et le premier terme non trivial de la suite exacte \eqref{SEext} est nul d'après \cite[Proposition 5.1.4 (i)]{JH}, donc on obtient le point (i) modulo $p$.
Si $\chi'=\chi$, alors on obtient les bornes du point (ii) de la remarque \ref{rema:ext} modulo $p$. En particulier, on en déduit le point (ii) modulo $p$.

On prouve maintenant le théorème. Soient $\chi,\chi' : T(\Qp) \to \Oe^\times \subset E^\times$ des caractères continus unitaires. Pour $k \geq 1$ entier, les suites exactes \eqref{SEext} et \eqref{SEexta} et le diagramme de la proposition \ref{prop:compat} avec $A=\A{k}$ et les réductions modulo $\pe^k$ des caractères $\chi$ et $\chi'$ forment des systèmes projectifs. On passe à la limite projective puis on tensorise par $E$ sur $\Oe$ en utilisant \cite[Lemme 4.1.3]{Em1} et \cite[Proposition B.2]{JH}. On utilise \cite[Proposition 5.1.6]{JH} au lieu de \cite[Proposition 5.1.4 (i)]{JH}. Le reste de la démonstration est identique à la version modulo $p$. 
\end{proof}

On détermine toutes les auto-extensions d'une série principale de $G(\Qp)$ modulo $p$ lorsque le centre de $G$ est connexe et on en déduit le résultat analogue $p$-adique lorsque $p \neq 2$.

\begin{theo} \label{theo:autoextmodp}
On suppose le centre de $G$ connexe. Soit $\chi : T(\Qp) \to \ke^\times$ un caractère lisse.
\begin{enumerate}[(i)]
\item Si $p\neq2$, alors le foncteur $\IndQp$ induit un isomorphisme $\ke$-linéaire
\begin{equation*}
\Ext_{T(\Qp)}^1 \left( \chi,\chi \right) \iso \Ext_{G(\Qp)}^1 \left( \IndQp \chi,\IndQp \chi \right).
\end{equation*}
\item Si $p=2$, alors le foncteur $\IndQp$ induit une injection $\ke$-linéaire
\begin{equation*}
\Ext_{T(\Qp)}^1 \left( \chi,\chi \right) \hookrightarrow \Ext_{G(\Qp)}^1 \left( \IndQp \chi,\IndQp \chi \right)
\end{equation*}
dont le conoyau est de dimension $\card \{\alpha \in \Delta \mid s_\alpha(\chi)=\chi\}$.
\end{enumerate}
\end{theo}

\begin{proof}
La suite exacte \eqref{SEext} avec $A=\ke$ et $\chi'=\chi$ se complète en une suite exacte de $\ke$-espaces vectoriels
\begin{multline} \label{SEGroth}
0 \to \Ext^1_{T(\Qp)} \left( \chi,\chi \right) \to \Ext^1_{G(\Qp)} \left( \IndQp \chi,\IndQp \chi \right) \\
\to \Hom_{T(\Qp)} \left( \chi,\ROrdQp[1] \left( \IndQp \chi \right) \right) \to \Ext^2_{T(\Qp)} \left( \chi,\chi \right) \\
\to \Ext^2_{G(\Qp)} \left( \IndQp \chi,\IndQp \chi \right)
\end{multline}
dont le premier morphisme non trivial et le dernier morphisme sont induits par le foncteur $\IndQp$ et où $\ROrdQp$ désigne les foncteurs dérivés à droite du foncteur $\OrdQp$ dans la catégorie des représentations lisses localement admissibles de $G(\Qp)$ sur $\ke$ (voir la suite exacte \cite[(3.7.5)]{Em2}).
En utilisant \cite[Remarque 3.7.3]{Em2} et l'isomorphisme \eqref{isoHOrd}, on voit que le troisième terme non trivial de cette suite exacte est de dimension au plus $\card \Delta'$ avec
\begin{equation*}
\Delta' \dfn \left\{ \alpha \in \Delta \mid s_\alpha(\chi) \cdot (\oma) = \chi \right\}.
\end{equation*}
Si $p=2$, alors en utilisant le lemme \ref{lemm:extamod2} (voir aussi la remarque \ref{rema:irregmod2}) et en procédant comme dans la preuve de la version modulo $p$ du théorème \ref{theo:ext}, on obtient le point (ii).

On suppose $p \neq 2$ et on montre le point (i).
Comme $Z$ est connexe, il existe des sous-tores $(T_\alpha)_{\alpha \in \Delta}$ de $T$ tels que l'on a un isomorphisme $T \cong Z \times \prod_{\alpha \in \Delta} T_\alpha$ à travers lequel $Z_\beta \cong Z \times \prod_{\alpha \in \Delta-\{\beta\}} T_\alpha$ pour tout $\beta \in \Delta$ (il suffit de prendre les images de copoids fondamentaux $(\lambda_\alpha)_{\alpha \in \Delta}$, voir \cite[Proposition 2.1.1]{BH}).
Pour tous $\alpha,\beta \in \Delta$, la composée des morphismes naturels $\ke$-linéaires
\begin{equation*}
\Ext^\bullet_{T(\Qp),\chi_{|\Za(\Qp)}} \left( \chi,\chi \right) \to \Ext^\bullet_{T(\Qp)} \left( \chi,\chi \right) \to \Ext^\bullet_{T_\beta(\Qp)} \left( \chi,\chi \right)
\end{equation*}
est un isomorphisme si $\alpha=\beta$ (car $T \cong T_\alpha \times \Za$) et elle est nulle sinon (car $T_\beta \subset \Za$).
Ainsi, le morphisme naturel $\ke$-linéaire
\begin{equation} \label{somext}
\bigoplus_{\alpha \in \Delta} \Ext^\bullet_{T(\Qp),\chi_{|\Za(\Qp)}} \left( \chi,\chi \right) \to \Ext^\bullet_{T(\Qp)} \left( \chi,\chi \right)
\end{equation}
est injectif.
Pour tout $\alpha \in \Delta'$, les lemmes \ref{lemm:irreg} et \ref{lemm:exta2} montrent qu'il existe une auto-extension non triviale $\Ea$ de longueur $2$ de $\chi$ sur laquelle $\Za(\Qp)$ agit à travers $\chi_{|\Za(\Qp)}$ et telle que $\Inda \Ea$ est une auto-extension triviale de longueur $2$ de $\Inda \chi$.
Les images par le morphisme \eqref{somext} des classes des auto-extensions $(\Ea)_{\alpha \in \Delta'}$ engendrent donc un sous-$\ke$-espace vectoriel de dimension $\card \Delta'$ qui est par construction dans le noyau du dernier morphisme de la suite exacte \eqref{SEGroth}, donc le premier morphisme non trivial de la suite exacte \eqref{SEGroth} est un isomorphisme.
\end{proof}

\begin{rema} \phantomsection \label{rema:autoextmodp}
\begin{enumerate}[(i)]
\item En général (sans supposer le centre de $G$ connexe) si $p \neq 2$, alors en posant
\begin{equation*}
\Delta'' \dfn \left\{ \alpha \in \Delta' \mid \chi \circ \alpha^\vee = \omega^{-1} \right\},
\end{equation*}
on peut construire des auto-extensions $(\Ea)_{\alpha \in \Delta''}$ comme dans la preuve (voir la remarque \ref{rema:exta2irreg}) et montrer que les images par le morphisme \eqref{somext} de leurs classes engendrent un sous-$\ke$-espace vectoriel de dimension $\card \Delta'' - \dim_{\Fp} \Hom(Z''/Z''^\circ,\mu_p)$ avec $Z'' \subset T$ le sous-groupe $\bigcap_{\alpha \in \Delta''} \ker \alpha$ et $\mu_p$ le groupe des racines $p$-ièmes de l'unité, donc on s'attend à ce que le conoyau de l'injection $\ke$-linéaire induite par le foncteur $\IndQp$ soit de dimension
\begin{equation*}
\card \left( \Delta' - \Delta'' \right) + \dim_{\Fp} \Hom \left( Z''/Z''^\circ,\mu_p \right).
\end{equation*}
\item Le point (ii) est encore vrai sans supposer le centre de $G$ connexe (cette hypothèse n'est pas utilisée dans la preuve).
\end{enumerate}
\end{rema}

\begin{coro} \label{coro:autoext}
On suppose le centre de $G$ connexe et $p \neq 2$.
Pour tout caractère continu unitaire $\chi : T(\Qp) \to \Oe^\times \subset E^\times$, le foncteur $\IndQp$ induit un isomorphisme $E$-linéaire
\begin{equation*}
\Ext_{T(\Qp)}^1 \left( \chi,\chi \right) \iso \Ext_{G(\Qp)}^1 \left( \IndQp \chi,\IndQp \chi \right).
\end{equation*}
\end{coro}

\begin{proof}
Soit $\chi : T(\Qp) \to \Oe^\times \subset E^\times$ un caractère continu unitaire.
Le foncteur exact $\IndQp$ admet un quasi-inverse à gauche (voir \cite[Corollaire 4.3.5]{Em1}), donc il induit une injection $E$-linéaire
\begin{equation*}
\Ext_{T(\Qp)}^1 \left( \chi,\chi \right) \hookrightarrow \Ext_{G(\Qp)}^1 \left( \IndQp \chi,\IndQp \chi \right).
\end{equation*}
En utilisant l'inégalité de \cite[Proposition B.2]{JH} et le fait que $p \neq 2$ d'où
\begin{equation*}
\dim_{\ke} \Ext_{T(\Qp)}^1 \left( \overline{\chi},\overline{\chi} \right) = \dim_E \Ext_{T(\Qp)}^1 \left( \chi,\chi \right)
\end{equation*}
avec $\overline{\chi} : T(\Qp) \to \ke^\times$ la réduction modulo $\pe$ de $\chi$ (voir \cite[Proposition 5.1.4 (ii)]{JH} et \cite[Proposition 5.1.6]{JH}), on déduit du point (i) du théorème \ref{theo:autoextmodp} que cette injection est un isomorphisme.
\end{proof}

\begin{rema}
On s'attend à ce que le résultat soit encore vrai lorsque $p=2$.
\end{rema}

\subsection{Variante pour les représentations ordinaires}

Soient $P \subset G$ un sous-groupe parabolique standard et $L \subset P$ le sous-groupe de Levi standard. On reprend les notations de la sous-section \ref{ssec:ord}.

\begin{defi} \label{defi:Spcont}
Pour tout sous-groupe parabolique standard $Q \subset L$, on définit la \emph{représentation spéciale} relative à $Q$ de $L(F)$ sur $E$
\begin{equation*}
\Spc_Q \dfn \frac{\IndL[Q^-] \unc}{\sum_{Q \subsetneqq Q' \subset L} \IndL[Q'^-] \unc}
\end{equation*}
avec $Q'$ parmi les sous-groupes paraboliques de $L$ et $\unc$ la représentation triviale sur $E$. Lorsque $Q=B_L$, on l'appelle la \emph{représentation de Steinberg} de $L(F)$ sur $E$ et on la note $\Stc$. Lorsque $Q=Q_\alpha$ avec $\alpha \in \Delta_L$, on la note $\Spc_\alpha$.
\end{defi}

\begin{prop} \label{prop:Spc}
Les représentations spéciales sont des représentations continues unitaires admissibles de $L(F)$ sur $E$ topologiquement irréductibles et deux à deux non isomorphes. Elles forment les constituants irréductibles de $\IndL \unc$, chacune apparaissant avec multiplicité un.
\end{prop}

Cette proposition est une conséquence directe du résultat analogue modulo $p$ (voir la remarque \ref{rema:spe}) et du lemme suivant avec $k=1$.

\begin{lemm} \label{lemm:spe}
Soit $Q \subset L$ un sous-groupe parabolique standard.
Il existe une boule $\Spcz_Q \subset \Spc_Q$ stable par $L(F)$ telle que pour tout entier $k \geq 1$, $\Spcz_Q/\pe^k\Spcz_Q$ est isomorphe à la représentation spéciale relative à $Q$ de $L(F)$ sur $\A{k}$.
\end{lemm}

\begin{proof}
Pour tout $r \in \Z$, on note $\IndL[Q^-] \pe^r\Oe \subset \IndL[Q^-] E$ la boule stable par $L(F)$ constituée des fonctions à valeurs dans $\pe^r\Oe$. On note $\Spcz_Q$ l'image de $\IndL[Q^-] \Oe$ dans $\Spc_Q$. C'est une boule de $\Spc_Q$ stable par $L(F)$ et on a une suite exacte courte
\begin{equation*}
0 \to \sum_{Q \subsetneqq Q' \subset L} \IndL[Q'^-] \Oe \to \IndL[Q^-] \Oe \to \Spcz_Q \to 0
\end{equation*}
(l'injectivité du premier morphisme non trivial est immédiate, la surjectivité du second morphisme non trivial résulte de la définition de $\Spcz_Q$ et l'exactitude au milieu se déduit du sous-lemme ci-dessous avec $r=0$).

Soit $k \geq 1$ entier. La composition avec la projection $\Oe \twoheadrightarrow \A{k}$ induit une application
\begin{equation*}
\sum_{Q \subsetneqq Q' \subset L} \IndL[Q'^-] \Oe \twoheadrightarrow \sum_{Q \subsetneqq Q' \subset L} \IndL[Q'^-] \A{k}
\end{equation*}
(où la première somme est calculée dans $\IndL[Q^-] \Oe$ et la seconde dans $\IndL[Q^-] \A{k}$) qui coïncide avec la réduction modulo $\pe^k$ (cela résulte du sous-lemme ci-dessous avec $r=k$).
On en déduit que la réduction modulo $\pe^k$ (c'est-à-dire l'image par le foncteur $\A{k} \otimes_{\Oe} (-)$) de la suite exacte précédente est la suite exacte courte
\begin{equation*}
0 \to \sum_{Q \subsetneqq Q' \subset L} \IndL[Q'^-] \A{k} \to \IndL[Q^-] \A{k} \to \Spcz_Q/\pe^k\Spcz_Q \to 0
\end{equation*}
(l'injectivité du premier morphisme non trivial résulte du fait que $\Spcz_Q$ est un $\Oe$-module sans torsion donc plat). Ainsi, $\Spcz_Q/\pe^k\Spcz_Q$ est isomorphe à la représentation spéciale relative à $Q$ de $L(F)$ sur $\A{k}$.
\end{proof}

\begin{slem}
Soient $n \in \Nbb$ et $Q \subsetneqq Q_1,\dots,Q_n \subset L$ des sous-groupes paraboliques standards. Pour tout $r \in \Z$, on a l'égalité suivante dans $\IndL[Q^-] E$ :
\begin{equation*}
\IndL[Q^-] \pe^r\Oe \cap \sum_{k=1}^n \IndL[Q_k^-] E = \sum_{k=1}^n \IndL[Q_k^-] \pe^r\Oe.
\end{equation*}
\end{slem}

\begin{proof}
On procède par récurrence sur $n$. Le cas $n=0$ est trivial. On suppose $n>0$ et le lemme vrai pour $n-1$. En multipliant par $\pe^{-r}$, on se ramène au cas $r=0$. Il suffit de montrer que pour tout $l \in \Nbb$, on a l'inclusion suivante dans $\IndL[Q^-] E$ :
\begin{equation*}
\IndL[Q^-] \Oe \cap \sum_{k=1}^n \IndL[Q_k^-] \pe^{-l}\Oe \subset \sum_{k=1}^n \IndL[Q_k^-] \Oe.
\end{equation*}

On procède par récurrence sur $l$. Le cas $l=0$ est trivial. On suppose $l>0$ et l'inclusion vraie pour $l-1$.
Soit $(f_k)_{k \in \llbrack 1,n \rrbrack} \in \prod_{k=1}^n \IndL[Q_k^-] \pe^{-l}\Oe$ tel que $\sum_{k=1}^n f_k \in \IndL[Q^-] \Oe$.
Pour tout $k \in \llbrack 1,n \rrbrack$, on note $\overline{f}_k \in \IndL[Q_k^-] \ke$ la composée de $f_k$ avec la projection $\pe^{-l}\Oe \twoheadrightarrow \pe^{-l}\Oe/\pe^{1-l}\Oe \cong \ke$. On a donc l'égalité $\sum_{k=1}^n \overline{f}_k=0$ dans $\IndL[Q^-] \ke$. Comme $\IndL[Q^-] \ke$ est de longueur fini sans multiplicité dans la catégorie des représentations lisses de $L(F)$ sur $\ke$ (voir la remarque \ref{rema:spe}), on a l'égalité suivante dans $\IndL[Q^-] \ke$ :
\begin{equation*}
\IndL[Q_n^-] \ke \cap \sum_{k=1}^{n-1} \IndL[Q_k^-] \ke = \sum_{k=1}^{n-1} \left( \IndL[Q_n^-] \ke \cap \IndL[Q_k^-] \ke \right).
\end{equation*}
Ainsi, $\overline{f}_n = \sum_{k=1}^{n-1} \overline{f}{}'_k$ avec $\overline{f}{}'_k \in \IndL[Q_n^-] \ke \cap \IndL[Q_k^-] \ke$ pour tout $k \in \llbrack 1,n-1 \rrbrack$. On choisit un relèvement $f_k' \in \IndL[Q_n^-] \pe^{-l}\Oe \cap \IndL[Q_k^-] \pe^{-l}\Oe$ de $\overline{f}{}'_k$ pour tout $k \in \llbrack 1,n-1 \rrbrack$. On obtient $f_n = \sum_{k=1}^{n-1} f_k' + f_n'$ avec $f_n' \in \IndL[Q_n^-] \pe^{1-l}\Oe$, donc $\sum_{k=1}^n f_k = \sum_{k=1}^{n-1} (f_k+f_k') + f_n' \in \IndL[Q^-] \Oe$ par hypothèse, d'où $\sum_{k=1}^{n-1} (f_k+f_k') \in \IndL[Q^-] \pe^{1-l}\Oe$.
Par hypothèse de récurrence avec $n-1$, on en déduit que $\sum_{k=1}^{n-1} (f_k+f_k') = \sum_{k=1}^{n-1} f_k''$ avec $f_k'' \in \IndL[Q_k^-] \pe^{1-l}\Oe$ pour tout $k \in \llbrack 1,n-1 \rrbrack$ et par hypothèse de récurrence avec $l-1$, on en conclut que $\sum_{k=1}^n f_k = \sum_{k=1}^{n-1} f_k'' + f_n' \in \sum_{k=1}^n \IndL[Q_k^-] \Oe$.
\end{proof}

\begin{defi}
Une \emph{représentation ordinaire} de $G(F)$ sur $E$ est une représentation de la forme $\Ind[P^-] (\sigma \otimes \chi)$ avec $\sigma$ une représentation spéciale de $L(F)$ sur $E$ et $\chi : L(F) \to \Oe^\times \subset E^\times$ un caractère continu unitaire.
\end{defi}

\begin{rema}
Si \cite[Conjecture 3.1.2]{BH} est vraie, alors les représentations ordinaires topologiquement irréductibles sont exactement les sous-quotients irréductibles de séries principales (cela se déduit du résultat analogue modulo $p$).
\end{rema}

On détermine les extensions d'une série principale de $G(F)$ par une représentation ordinaire de $G(F)$. On traite séparément les cas $F=\Qp$ et $F \neq \Qp$.

\begin{prop} \label{prop:extord}
On suppose $F=\Qp$. Soient $\sigma$ une représentation spéciale de $L(\Qp)$ sur $E$ et $\chi : L(\Qp) \to \Oe^\times \subset E^\times$ un caractère continu unitaire.
\begin{enumerate}[(i)]
\item Si $\chi' : T(\Qp) \to \Oe^\times \subset E^\times$ est un autre caractère continu unitaire, alors
\begin{equation*}
\Ext^1_{G(\Qp)} \left( \IndQp \chi',\IndQp[P^-] (\sigma \otimes \chi) \right) \neq 0
\end{equation*}
si et seulement si ou bien $\sigma=\Stc$ et $\chi' = \chi$, ou bien $\sigma=\Stc$ et $\chi' = s_\alpha(\chi) \cdot (\epsa)$ avec $\alpha \in \Delta-\Delta_L$, ou bien $\sigma=\Spc_\alpha$ et $\chi' = \chi \cdot (\epsa)$ avec $\alpha \in \Delta_L$.
\item Si $\chi' : T(\Qp) \to \Oe^\times \subset E^\times$ est un caractère continu unitaire distinct de $\chi$, alors
\begin{multline*}
\dim_E \Ext^1_{G(\Qp)} \left( \IndQp \chi',\IndQp[P^-] (\Stc \otimes \chi) \right) \\
= \card \left\{ \alpha \in \Delta-\Delta_L \mid \chi' = s_\alpha(\chi) \cdot (\epsa) \right\}.
\end{multline*}
\item Pour tout $\alpha \in \Delta_L$, on a
\begin{equation*}
\dim_E \Ext^1_{G(\Qp)} \left( \IndQp \chi \cdot (\epsa),\IndQp[P^-] (\Spc_\alpha \otimes \chi) \right) = 1.
\end{equation*}
\item Si $s_\alpha(\chi) \cdot (\epsa) \neq \chi$ pour tout $\alpha \in \Delta-\Delta_L$, alors le foncteur $\IndQp$ et la projection $\IndQp \chi \twoheadrightarrow \IndQp[P^-] (\Stc \otimes \chi)$ induisent un isomorphisme $E$-linéaire
\begin{equation*}
\Ext^1_{T(\Qp)} \left( \chi,\chi \right) \iso \Ext^1_{G(\Qp)} \left( \IndQp \chi,\IndQp[P^-] (\Stc \otimes \chi) \right).
\end{equation*}
\end{enumerate}
\end{prop}

\begin{rema} \label{rema:extord}
Sans l'hypothèse de généricité dans le point (iv), on montre que le foncteur $\IndQp$ et la projection $\IndQp \chi \twoheadrightarrow \IndQp[P^-] (\Stc \otimes \chi)$ induisent une injection $E$-linéaire dont le conoyau est de dimension comprise entre
\begin{gather*}
\card \left\{ \alpha \in \Delta - \Delta_L \mid s_\alpha(\chi) \cdot (\epsa) = \chi ~ \text{et} ~ \chi \circ \alpha^\vee \neq \varepsilon^{-1} \right\} \\
\text{et} ~ \card \left\{ \alpha \in \Delta - \Delta_L \mid s_\alpha(\chi) \cdot (\epsa) = \chi \right\}.
\end{gather*}
On s'attend à ce que la minoration soit une égalité (on peut déduire cela du point (i) du théorème \ref{theo:autoextmodp} lorsque le centre de $G$ est connexe et $p \neq 2$, auquel cas le minorant est nul d'après le lemme \ref{lemm:irreg}).
De même pour son analogue modulo $p$, sauf dans certains cas exceptionnels et lorsque $p=2$.
\end{rema}

\begin{proof}
Soient $\sigma$ une représentation spéciale de $L(\Qp)$ sur $A$ et $\chi : L(\Qp) \to A^\times$, $\chi' : T(\Qp) \to A^\times$ des caractères lisses. En utilisant le corollaire \ref{coro:ord} et le point (i) du corollaire \ref{coro:h1ord}, on déduit de la suite exacte \eqref{SEextP} pour le triplet $(G,B,T)$ avec $U=\chi'$ et $V = \IndQp[P^-] (\sigma \otimes \chi)$ que :
\begin{itemize}
\item si $\sigma \neq \St$ et $\sigma \neq \Sp_\alpha$ pour tout $\alpha \in \Delta_L$, alors
\begin{equation} \label{SEextSp}
\Ext^1_{G(\Qp)} \left( \IndQp \chi',\IndQp[P^-] (\sigma \otimes \chi) \right) = 0~;
\end{equation}
\item si $\sigma=\St$, alors on a une suite exacte de $A$-modules
\begin{multline} \label{SEextSt}
0 \to \Ext^1_{T(\Qp)} \left( \chi',\chi \right) \to \Ext^1_{G(\Qp)} \left( \IndQp \chi',\IndQp[P^-] (\St \otimes \chi) \right) \\
\to \bigoplus_{\alpha \in \Delta-\Delta_L} \Hom_{T(\Qp)} \left( \chi',s_\alpha(\chi) \cdot (\oma) \right) ;
\end{multline}
\item si $\sigma=\Sp_\alpha$ avec $\alpha \in \Delta_L$, alors on a une injection $A$-linéaire
\begin{multline} \label{SEextSpa}
\Ext^1_{G(\Qp)} \left( \IndQp \chi',\IndQp[P^-] (\Sp_\alpha \otimes \chi) \right) \\
\hookrightarrow \Hom_{T(\Qp)} \left( \chi',\chi \cdot (\oma) \right).
\end{multline}
\end{itemize}

On suppose $A=\ke$ et on prouve la version modulo $p$ de la proposition. En utilisant \cite[Proposition 5.1.4 (i)]{JH}, on obtient les cas d'annulation du point (i) modulo $p$. Les cas de non annulation résultent des autres points modulo $p$.
On pose
\begin{gather*}
\Delta' \dfn \left\{ \alpha \in \Delta \mid \chi' = s_\alpha(\chi) \cdot (\oma) \right\}, \\
\Delta'' \dfn \left\{ \alpha \in \Delta' \mid \chi \circ \alpha^\vee = \omega^{-1} \right\}.
\end{gather*}
Pour tout $\alpha \in \Delta_L$, on a une suite exacte de $\ke$-espaces vectoriels
\begin{multline} \label{SEextLSpa}
\Ext_{G(\Qp)}^1 \left( \IndQp \chi',\IndQp[P^-] (I_\alpha \otimes \chi) \right) \\
\to \Ext_{G(\Qp)}^1 \left( \IndQp \chi',\IndQp[\Pa^-] \chi \right) \\
\to \Ext_{G(\Qp)}^1 \left( \IndQp \chi',\IndQp[P^-] (\Sp_\alpha \otimes \chi) \right)
\end{multline}
avec $I_\alpha$ le noyau de la projection $\IndQpL[Q_\alpha^-] \un \twoheadrightarrow \Sp_\alpha$.
D'une part les constituants irréductibles de $I_\alpha$ sont les représentations $(\Sp_Q)_{Q_\alpha \subsetneqq Q \subset L}$ avec $Q$ parmi les sous-groupes paraboliques de $L$ (voir la remarque \ref{rema:spe}), donc on déduit de l'égalité \eqref{SEextSp} avec $\sigma=\Sp_Q$ pour tout sous-groupe parabolique $Q \subset L$ tel que $Q_\alpha \subsetneqq Q$ que le premier terme de la suite exacte \eqref{SEextLSpa} est nul.
D'autre part le foncteur exact $\IndQp[\Pa^-]$ admet un quasi-inverse à gauche d'après le lemme \ref{lemm:inj} pour le triplet $(G,\Pa,\Ga)$, donc il induit une injection $\ke$-linéaire
\begin{equation*}
\Ext_{\Ga(\Qp)}^1 \left( \Inda \chi',\chi \right) \hookrightarrow \Ext_{G(\Qp)}^1 \left( \IndQp \chi',\IndQp[\Pa^-] \chi \right)
\end{equation*}
et en utilisant l'inégalité \eqref{extaSpa} on en déduit que le terme du milieu de la suite exacte \eqref{SEextLSpa} est non nul lorsque $\alpha \in \Delta'$.
On en conclut que le dernier terme de la suite exacte \eqref{SEextLSpa} est non nul lorsque $\alpha \in \Delta'$ et en utilisant l'injection \eqref{SEextSpa}, on obtient le point (iii) modulo $p$.
En utilisant la suite exacte \eqref{SEextSt}, on voit que le conoyau du premier morphisme non trivial est de dimension au plus $\card (\Delta - \Delta_L) \cap \Delta'$.
La représentation $\IndQp \chi \cong \IndQp[P^-] ((\IndQpL[B_L^-] \un) \otimes \chi)$ admet une filtration dont les quotients successifs sont exactement les représentations $(\IndQp[P^-] (\Sp_Q \otimes \chi))_{Q \subset L}$ avec $Q$ parmi les sous-groupes paraboliques standards de $L$ (voir la remarque \ref{rema:spe}). En utilisant le point (i) et la minoration du point (ii) de l'analogue modulo $p$ du théorème \ref{theo:ext}, on voit que
\begin{multline*}
\dim_{\ke} \Ext^1_{T(\Qp)} \left( \chi',\chi \right) + \card \left( \Delta'-\Delta'' \right) \\
\leq \sum_{Q \subset L} \dim_{\ke} \Ext^1_{G(\Qp)} \left( \IndQp \chi',\IndQp[P^-] (\Sp_Q \otimes \chi) \right)
\end{multline*}
avec $Q$ parmi les sous-groupes paraboliques standards de $L$. En utilisant les cas d'annulation du point (i) modulo $p$ et le point (iii) modulo $p$, on en déduit que
\begin{multline*}
\dim_{\ke} \Ext^1_{T(\Qp)} \left( \chi',\chi \right) + \card \left( \left( \Delta - \Delta_L \right) \cap \left( \Delta' - \Delta'' \right) \right) \\
\leq \dim_{\ke} \Ext^1_{G(\Qp)} \left( \IndQp \chi',\IndQp[P^-] (\St \otimes \chi) \right),
\end{multline*}
donc le conoyau du premier morphisme non trivial de la suite exacte \eqref{SEextSt} est de dimension au moins $\card ((\Delta-\Delta_L) \cap (\Delta'-\Delta''))$.
Si $\chi' \neq \chi$, alors $\Delta''=\emptyset$ d'après le lemme \ref{lemm:irreg} et le premier terme non trivial de la suite exacte \eqref{SEextSt} est nul d'après \cite[Proposition 5.1.4 (i)]{JH}, donc on obtient le point (ii) modulo$p$.
Si $\chi'=\chi$, alors on obtient les bornes de la remarque \ref{rema:extord} modulo $p$. En particulier, on en déduit le point (iv) modulo $p$.

On prouve maintenant la proposition. Soient $\sigma$ une représentation spéciale de $L(\Qp)$ sur $E$ et $\chi : L(\Qp) \to \Oe^\times \subset E^\times$, $\chi' : T(\Qp) \to \Oe^\times \subset E^\times$ des caractères continus unitaires. Pour $k \geq 1$ entier, l'égalité \eqref{SEextSp}, la suite exacte \eqref{SEextSt} et l'injection \eqref{SEextSpa} avec $A=\A{k}$ et les réductions modulo $\pe^k$ de la représentation spéciale $\sigma$ et des caractères $\chi$ et $\chi'$ forment des systèmes projectifs. On passe à la limite projective puis on tensorise par $E$ sur $\Oe$ en utilisant le lemme \ref{lemm:spe}, \cite[Lemme 4.1.3]{Em1} et \cite[Proposition B.2]{JH}. On utilise \cite[Proposition 5.1.6]{JH} au lieu de \cite[Proposition 5.1.4 (i)]{JH}. Le reste de la démonstration est identique à la version modulo $p$.
\end{proof}

\begin{prop} \label{prop:extordF}
On suppose $F \neq \Qp$. Soient $\sigma$ une représentation spéciale de $L(F)$ sur $E$ et $\chi : L(F) \to \Oe^\times \subset E^\times$ un caractère continu unitaire.
\begin{enumerate}[(i)]
\item Si $\chi' : T(F) \to \Oe^\times \subset E^\times$ est un autre caractère continu unitaire, alors
\begin{equation*}
\Ext^1_{G(F)} \left( \Ind \chi',\Ind[P^-] (\sigma \otimes \chi) \right) \neq 0
\end{equation*}
si et seulement si $\sigma = \Stc$ et $\chi'=\chi$.
\item Le foncteur $\Ind$ et la projection $\Ind \chi \twoheadrightarrow \Ind[P^-] (\Stc \otimes \chi)$ induisent un isomorphisme $E$-linéaire
\begin{equation*}
\Ext^1_{T(F)} \left( \chi,\chi \right) \iso \Ext^1_{G(F)} \left( \Ind \chi,\Ind[P^-] (\Stc \otimes \chi) \right).
\end{equation*}
\end{enumerate}
\end{prop}

\begin{proof}
Soient $\sigma$ une représentation spéciale de $L(F)$ sur $A$ et $\chi : L(F) \to A^\times$, $\chi' : T(F) \to A^\times$ des caractères lisses. En utilisant le corollaire \ref{coro:ord} et le point (ii) du corollaire \ref{coro:h1ord}, on déduit de la suite exacte \eqref{SEextP} pour le triplet $(G,B,T)$ avec $U=\chi'$ et $V = \Ind[P^-] (\sigma \otimes \chi)$ que :
\begin{itemize}
\item si $\sigma \neq \St$, alors
\begin{equation*}
\Ext^1_{G(F)} \left( \Ind \chi',\Ind[P^-] (\sigma \otimes \chi) \right) = 0~;
\end{equation*}
\item si $\sigma=\St$, alors on a un isomorphisme $A$-linéaire
\begin{equation*}
\Ext^1_{T(F)} \left( \chi',\chi \right) \iso \Ext^1_{G(F)} \left( \Ind \chi',\Ind[P^-] (\St \otimes \chi) \right).
\end{equation*}
\end{itemize}
Avec $A=\ke$ et en utilisant \cite[Proposition 5.1.4 (i)]{JH} on prouve la version modulo $p$ de la proposition. Par un passage à la limite projective et un produit tensoriel comme dans la preuve de la proposition \ref{prop:extord} et en utilisant \cite[Proposition 5.1.6]{JH} on prouve la proposition.
\end{proof}

\bibliographystyle{alpha-fr}
\bibliography{complements}

\end{document}